\documentclass[11pt]{amsart}
\usepackage{enumitem, amsmath, amsthm, amsfonts, amssymb, xy,  mathrsfs, graphicx, lscape}
\usepackage[margin=1in]{geometry}

\theoremstyle{plain}

\newtheorem{thm}{Theorem}[section]
\newtheorem*{thm*}{Theorem}
\newtheorem{cor}[thm]{Corollary}
\newtheorem{lem}[thm]{Lemma}
\newtheorem{claim}[thm]{Claim}
\newtheorem{prop}[thm]{Proposition}
\newtheorem{ques}[thm]{Question}
\newtheorem{prob}[thm]{Problem}
\newtheorem{conj}[thm]{Conjecture}
\theoremstyle{definition}

\newtheorem{defn}[thm]{Definition}

\theoremstyle{remark}
\newtheorem{rem}[thm]{Remark}

\newtheorem{rmk}[equation]{Remark}

\newtheorem{eg}[equation]{Example}

\newenvironment{subproof}[1][\proofname]{%
  \begin{proof}[#1]%
}{%
  \end{proof}%
}

\usepackage{bbm}

\newcommand{\Z}{\mathbb{Z}}
\newcommand{\C}{\mathbb{C}}
\newcommand{\N}{\mathbb{N}}

\newcommand{\ngroup}{\trianglelefteq}

\renewcommand{\phi}{\varphi}
\renewcommand{\emptyset}{\varnothing}

\renewcommand{\tilde}[1]{\widetilde{#1}}

\makeatletter
\def\Ddots{\mathinner{\mkern1mu\raise\p@
\vbox{\kern7\p@\hbox{.}}\mkern2mu
\raise4\p@\hbox{.}\mkern2mu\raise7\p@\hbox{.}\mkern1mu}}
\makeatother

\numberwithin{equation}{section}

\usepackage[utf8]{inputenc}

\usepackage[english]{babel}
\usepackage{csquotes}
\usepackage{mathrsfs}
\usepackage{comment}
\usepackage{tikz-cd}

\newcommand{\vardbtilde}[1]{\tilde{\raisebox{0pt}[0.85\height]{$\tilde{#1}$}}}

\newcommand{\Ast}{\mathop{\scalebox{1.5}{\raisebox{-0.2ex}{$\ast$}}}}
\newcommand{\Al}{\{ A_\lambda\}_{\lambda \in \Lambda}}
\newcommand{\NN}{\mathbb{N}}
\newcommand{\ZZ}{\mathbb{Z}}

\newcommand{\e}{\varepsilon}
\renewcommand{\ll }{\left\langle\hspace{-.7mm}\left\langle }
\newcommand{\rr }{\right\rangle\hspace{-.7mm}\right\rangle }
\newcommand{\Hl}{\{ H_\lambda\}_{\lambda \in \Lambda}}
\newcommand{\WR}{\mathcal{WR}}
\newcommand{\Ker}{\operatorname{Ker}}
\newcommand{\Aut}{\operatorname{Aut}}
\newcommand{\Inn}{\operatorname{Inn}}
\newcommand{\Out}{\operatorname{Out}}

\newcommand{\cM}{\mathcal{M}}

\newcommand{\cL}{\mathcal{L}}
\newcommand{\cN}{\mathcal{N}}

\newcommand{\cQ}{\mathcal{Q}}

\newcommand{\ra}{\rightarrow}

\theoremstyle{plain} 
\usepackage{hyperref}
\newtheorem*{genericthm*}{\thistheoremname}
\newcommand{\thistheoremname}{???}
\newcounter{genericthm}

\newenvironment{namedthm*}[1]
  {\renewcommand{\thistheoremname}{#1}%
   \refstepcounter{genericthm}%
   \begin{genericthm*}}
  {\end{genericthm*}}

\usepackage[
backend=biber,
style=alphabetic,
bibstyle=ieee-alphabetic,
sorting=nyt,
maxbibnames   = 99,
maxalphanames = 8,
maxcitenames  = 6,
minalphanames = 6,
minnames      = 4,
]{biblatex}
\title{W$^*$-superrigidity for Property (T) Groups with Infinite Center}
\author{Ionu\c t Chifan, Adriana Fernández Quero, Denis Osin, and Hui Tan}

\begin{document}

\begin{abstract} 
We propose to study a natural version of Connes' Rigidity Conjecture that involves property (T) groups with infinite center. Utilizing techniques at the intersection of von Neumann algebras and geometric group theory, we establish several cases where this conjecture holds. In particular, we provide the first example of a W$^*$-superrigid property (T) group with infinite center. In the course of proving our main results, we also generalize the main W$^*$-superrigidity result from \cite{cios22} to twisted group factors.
\end{abstract}

\maketitle

\vspace{-7mm}

\section{Introduction}

\noindent To any countable group $G$, one can associate its von Neumann algebra $\mathcal{L}(G)$ \cite{mvn43}, defined as the weak operator closure of the complex group algebra $\mathbb{C}[G]$, acting by left convolution on the Hilbert space $\ell^2G$ of square-summable functions on $G$. A major research theme in the field of von Neumann algebras, which has garnered considerable attention over the years, is understanding the extent to which $\mathcal{L}(G)$ retains algebraic information about the underlying group $G$. 

Recall that a group $G$ is said to have the \textit{ICC property} if the conjugacy class of every non-trivial element of $G$ is infinite. In \cite{connes80F} A. Connes discovered that II$_1$ factors associated with ICC Kazhdan property (T) groups exhibit strong rigidity under small perturbations; in particular, he showed that both the fundamental group and the outer automorphism group of every such II$_1$-factor is countable. These results and their fairly conceptual proofs motivated Connes to conjecture that \emph{if $G$,$H$ are ICC property (T) groups with $\mathcal L(G)\cong \mathcal L(H)$, then $G\cong H$}, \cite{connes80T}. As property (T) for groups is a von Neumann algebra invariant by \cite{connesjones}, his conjecture can be rephrased as
 
 \begin{namedthm*}{Connes' Rigidity Conjecture}
 \label{CRC} Any ICC property (T) group $G$  is W$^*$-superrigid; that is, whenever $H$ is an arbitrary group such that $\mathcal L(G)\cong \mathcal L(H)$ then $G\cong H$.
 \end{namedthm*}



Over the last two decades, remarkable progress in identifying classes of W$^*$-superrigid groups has been made using Popa's deformation/rigidity theory; see, for instance, \cite{popa06, vaes10, io18, ipv10, bv13, ci18, cdad20, cios22, DV24}, just to name a few. For instance, in \cite{cios22} were unveiled the first examples of groups satisfying Connes Rigidity Conjecture. However, almost all existing rigidity results focus on the case when $\mathcal{L}(G)$ has trivial center, which corresponds to $G$ being an ICC group.

The first W$^*$-rigidity results for von Neumann algebras $\mathcal L(G)$ groups $G$ with infinite center were obtained in \cite{cfqt23}, focusing on direct products $G = A \times K$, where $A$ is an infinite abelian group and $K$ is an ICC property (T) wreath-like product group as in \cite{cios22}. Such product groups cannot be W$^*$-superrigid as their center cannot be reconstructed from their von Neumann algebra. However, in \cite{cfqt23} it was shown that this is the only obstruction to their W$^*$-superrigidity: \emph{any group $H$ with $\mathcal{L}(G) \cong \mathcal{L}(H)$ must be a direct product of the form $H \cong B \times K$, where $B$ is an infinite abelian group}.

The main goal of this paper is to construct examples of property (T) groups with infinite center that are W$^*$-superrigid. Our approach builds on methods from \cite{cfqt23} and uses a generalization of the notion of wreath-like products of groups introduced in \cite{cios22}. These, along with the non-property (T) examples involving left-right wreath product groups obtained by Donvil and Vaes in parallel, independent work \cite{DV24center}, constitute the first known W$^*$-superrigid groups with an infinite center.


\vspace{2mm}

\noindent\textbf{Acknowledgments.} The authors thank Professors Adrian Ioana and Stefaan Vaes for their helpful comments and suggestions, which improved the exposition of this article. They are also grateful to the anonymous referee for valuable comments and suggestions, which greatly improved the exposition and overall quality of the paper.

\noindent Ionu\c t Chifan was supported by the NSF grants DMS-2154637 and FRG-DMS-1854194. Adriana Fernández Quero received support from the NSF grant DMS-2154637, the CLAS Dissertation Writing Fellowship and the Erwin and Peggy Kleinfeld Graduate Fellowship. Denis Osin was supported by the NSF grant DMS-2405032 and the Simons Fellowship in Mathematics MP-SFM-00005907. Hui Tan was supported in part by the NSF Grants DMS-1854074 and DMS-2153805.

\section{Main results}

\noindent Recall that two groups $G$ and $H$ are \emph{virtually isomorphic}  (denoted $G\cong_v H$) if one can find finite normal subgroups $F \ngroup G$ and $K\ngroup H$ such that in the corresponding quotients there are finite index subgroups $G_0\leqslant G/F$ and $H_0 \leqslant H/K$ such that $G_0$ is isomorphic to $H_0$. To properly state our results, we introduce the following versions of W$^*$-superrigidity.

\begin{defn}
We say that a group $G$ is \textit{virtually W$^*$-superrigid} if, for any group $H$ satisfying $\cL(G)\cong \cL(H)$, we have $G\cong_v H$. Further, we say that $G$ is \textit{strongly W$^*$-superrigid} if, for any group $H$ and any $\ast$-isomorphism  $\Psi\colon \cL(G)\ra \cL(H)$, there exist an isomorphism $\delta\colon G\to H$ and a unitary $w\in \cL(H)$ such that 

$$\Psi(u_g)= w v_{\delta(g)} w^*, \;\;\; \forall \, g\in G,$$

\noindent where $(u_g)_{g\in G}\subset \mathcal L(G)$ and  $(v_h)_{h\in H}\subset \mathcal L(H)$ are the canonical group unitaries. 

We note in passing that strongly W$^*$-superrigid groups $G$ automatically satisfy ${\rm Char}(G)=1$.

\end{defn}
For given groups $A$ and $B$ and an action $B\curvearrowright I$ on a set $I$,  we denote by $\WR(A,B\curvearrowright I)$ the class of \textit{wreath-like products of $A$ and $B$ corresponding to the action $B\curvearrowright I$} introduced in \cite{cios22}. Recall $W\in \WR(A,B\curvearrowright I)$ if there exists a short exact sequence  

$$1 \ra \oplus_{i\in I} A_i \hookrightarrow W \overset{\e}{\twoheadrightarrow} B\ra 1$$ 

\noindent such that $A_i\cong A$ for all $i\in I$ and $gA_i g^{-1} = A_{\e(g)\cdot i}$ for every $g\in W$, where $A_i$ is the $i$-th copy of $A$ in $\oplus_{i\in I} A_i=A^{(I)}$. We call the subgroup $A^{(I)}$ the \textit{base} of the wreath-like product $W$. When $I=B$, this class is denoted simply by $\WR(A,B)$ and its elements are called regular wreath-like products of $A$ and $B$. 

Our first result is the following.

\begin{namedthm*}{Theorem A}\label{TheoremA} Let $A$ be a nontrivial free abelian group, $B$ a nontrivial ICC subgroup of a hyperbolic group, $B\curvearrowright I$ an action on a countable set $I$ with amenable stabilizers. Let $G$ be a property (T) group with infinite center such that $1\to Z(G)\hookrightarrow G\twoheadrightarrow^{\pi} W\to 1$ is a central extension with $W\in \mathcal{WR}(A,B\curvearrowright I)$. Assume the extension splits over the base $A^{(I)}< W$. Then, $G$ is virtually W$^*$-superrigid.
\end{namedthm*}

In the above statement, saying that the short exact sequence splits over the base $A^{(I)}\leqslant W$ means that $\pi^{-1}(A^{(I)})=Z(G)\times A^{(I)}$, where $A^{(I)}$ is the base of the wreath-like product $W$.

We note that the virtual  W$^*$-superrigidity in the conclusion of \ref{TheoremA} cannot be promoted to the genuine W$^*$-superrigidity.  Indeed, let $Q$ be any infinite property (T) group and let $A$ and $B$ be any nonisomorphic finite abelian groups of the same order (e.g., we can take $A= \mathbb Z_{nm}$ and $B= \mathbb Z_n \times \mathbb Z_m$ for some positive integers $n,m\geq 2$ which are not co-prime). Consider $G= A \times Q$ and $H = B \times Q$. Since $G$ and $H$ are finite index extensions of $Q$ and the latter has property (T), it follows that $G$ and $H$ have property (T). Clearly, we have $G\ncong H$. However,

\begin{equation*}
    \cL(G)\cong \cL(A)\:\overline \otimes\: \cL(Q) \cong \mathbb C^{|A|}\:\overline \otimes\: \mathcal L(Q)\cong \mathbb C^{|B|}\:\overline \otimes\: \mathcal L(Q)\cong \cL( B)\:\overline \otimes \:\cL(Q)\cong \cL(H).
\end{equation*}

\noindent We note in passing that, as it will be shown in the proofs of our subsequent results, the lack of $\text{W}^*$-superrigidity for the groups $G$ in \ref{TheoremA} does not necessarily stem from the presence of torsion elements in $Z(G)$, as previous counterexamples might suggest. Rather, it arises from the fact that the image of the natural $2$-cocycle of $G$ fails to generate the entire center $Z(G)$. In the next W$^*$-superrigidity results, this does not happen because the groups involved have trivial abelianization.

As with the majority of previous rigidity results for group factors \cite{ipv10, bv13, ci18, cdad20, cios22, DV24}, it is desirable to provide a complete description of the $\ast$-isomorphism between $\mathcal{L}(G) \cong \mathcal{L}(H)$ in terms of the virtual isomorphism between the underlying groups, $G \cong_v H$, along with other relevant data about these groups, such as their multiplicative characters, or more generally, their finite-dimensional representations, inductions from their finite-index subgroups, etc. In Theorem \ref{mainresult}, we provide such a characterization, although it is somewhat too technical to present fully here. However, when $G$ satisfies certain natural additional conditions, the statement can be made more precise.

\vspace{2mm}

\begin{namedthm*}{Theorem B}\label{theoremB} Let $W\in\mathcal{WR}(A,B\curvearrowright I)$ and  $1\to Z(G)\hookrightarrow G\twoheadrightarrow W\to 1$ be groups as in the statement of \ref{TheoremA}. Assume in addition that $B$ is hyperbolic, and $G$ is torsion free and has trivial abelianization. Let $H$ be any group and let $\Theta\colon \cL(G)\to \cL(H)$ be any $*$-isomorphism. 

Then, we can find a finite family of projections $\{p_1,\ldots,p_n\}\subset \cL(Z(H))$ with $\sum_{i=1}^np_i=1$, finite subgroups $B_{i}<Z(H)$ and group monomorphisms $\delta_{i}\colon G\to H/B_{i}$ with finite index image, for each $1\leq i\leq n$, and a unitary $w\in \cL(H)$ such that 

\begin{equation*}
    \Theta(u_g)=w\left(\sum_{i=1}^nv_{\delta_i(g)}p_i\right)w^*,\text{ for all }g\in G.
\end{equation*}

\noindent Moreover, $\text{Im}(\delta_i)=[H,H]B_i/B_i$ and $[H,H]\cap B_i=[H,H]\cap B_j$ for all $1\leq i,j\leq n$.
\end{namedthm*}

Note that since $B_i\leqslant Z(H)$ and $p_i\in\mathscr Z(\cL(H))$ the product $v_{\delta_i(g)}p_i$ is independent of the choice of representative $\delta_i(g)\in H/B_i$, and therefore, it is a well-defined element of $\cL(H)$.

Our new quotienting technique for generalized wreath-like products (Theorem \ref{Thm:CenExt}), combined with methods from geometric group theory \cite{Osin07, DGO11, Sun, cios22b}, yields the following generalization of \cite[Theorem 2.7]{cios22b} to central extensions. This, in turn, provides many examples of property (T) groups $G$ that satisfy \ref{theoremB}.

\begin{namedthm*}{Theorem C}\label{TheoremD} 
For any finitely generated abelian groups $A$ and $C$, there is a group $G$ satisfying the following conditions:

\begin{enumerate}
    \item[(a)] $G$ is a central extension of the form $1\to C\to G\to^{\pi} W\to 1$, where $W\in \WR(A, B)$ for some non-elementary, torsion-free, hyperbolic group $B$; 
    \item[(b)] $\pi^{-1}(A^{(I)})=C\times A^{(I)}$, where $A^{(I)}\leqslant W$ is the base; and
    \item[(c)] $G$ has property (T) and trivial abelianization.
\end{enumerate}
\end{namedthm*}

We note that if one has a good control over the number of outer automorphisms of the group $G$ from \ref{theoremB}, then the conclusion of the theorem can be significantly strengthened. For example, we obtain the following.

\begin{namedthm*}{Corollary D}\label{TheoremC}  Let $W\in\mathcal{WR}(A,B\curvearrowright I)$ and $1\to Z(G)\hookrightarrow G\twoheadrightarrow W\to 1$ be groups  as in the statement of \ref{TheoremA}. Additionally, assume that $G$ has trivial abelianization and ${\rm Out}(G)=\{1\}$. Then $G$ is strongly W$^*$-superrigid.
\end{namedthm*}

Using an approach that combines our construction of central extensions with the control of outer automorphisms of Dehn fillings from \cite{DG18} (as also utilized in \cite[Theorem 2.10]{cios3}), we were able to obtain the following upgrade of \ref{TheoremD}, which, in particular, provides many examples of property (T) groups $G$ that satisfy \ref{TheoremC}.

\begin{namedthm*}{Theorem  E}\label{centralextwithOut=1} For each $n\in\N$, there is a central extension $1\to Z(G)\hookrightarrow G\twoheadrightarrow^{\pi} W\to 1$ satisfying

\begin{enumerate}
    \item[(a)] $Z(G)\cong \Z^n$;
    \item[(b)] $W\in\mathcal{WR}(\Z,B\curvearrowright I)$  is an ICC group where $B$ is ICC hyperbolic, $B\curvearrowright I$ is transitive with finite stabilizers and ${\rm Out}(B)=\{1\}$; 
    \item[(c)] $\pi^{-1}(\mathbb Z^{(I)})=Z(G)\times \mathbb Z^{(I)}$, where $\mathbb Z^{(I)}\leqslant W$ is the base; and
    \item[(d)] $G$ has property (T), trivial abelianization and ${\rm Out}(G)=\{1\}$. 
\end{enumerate}
\end{namedthm*}

We also mention in passing that the previous two results yield, in particular, property (T) groups $G$ with infinite centers extensions  for which ${\rm Out}(\mathcal{L}(G)) = 1$. To the best of our knowledge, this is the first concrete computation of outer automorphisms of property (T) von Neumann algebras with diffuse center.

Our main rigidity results for von Neumann algebras $\mathcal L(G)$ of groups with infinite center are obtained using a von Neumann algebraic technique that involves two main steps. Specifically, if $H$ is any group with $\mathcal L(H) \cong \mathcal L(G)$, the first step consists of reconstructing the center, namely, showing that, up to finite index and normal finite subgroup, $H$ is itself a nontrivial central extension with ICC central quotient $H/Z(H)$. This part is accomplished by largely recycling our previous methods from \cite{cfqt23}, including the techniques therein on integral decompositions of group von Neumann algebras. Once this is achieved, the second step, using integral decomposition techniques, is to establish W$^*$-superrigidity results for twisted group factors $\mathcal L_c(G/Z(G))$, which we view as one of the main contribution of this paper.

Specifically, we prove the following result, which generalizes the main result from \cite{cios22} and is motivated by a generalization of Popa's strengthening of Connes' rigidity conjecture \cite{popa07} to group II$_1$ factors twisted by scalar-valued 2-cocycles.

\begin{namedthm*}{Theorem F}\label{strongsuperrig} Let  $A$ be a nontrivial free abelian group, $B$ a nontrivial ICC subgroup of a hyperbolic group, $B\curvearrowright I$ an action on a countable set with amenable stabilizers. Let $G\in\mathcal{WR}(A,B\curvearrowright I )$ be a property (T) group and $d:G\times G\to\mathbb T$ a 2-cocycle that is trivial on the base $A^{(I)}$. Further, let $H$ be an arbitrary countable ICC group and let $c\colon H\times H\to\mathbb{T}$ be an arbitrary 2-cocycle. Assume that there exists a $*$-isomorphism. 
 $\Psi\colon \cL_d(G)^t\to \cL_c(H)$ for some $t>0$.

Then $t=1$ and there exists a group isomorphism $\delta\colon G\to H$ for which $d$ is cohomologous to $c\circ\delta$. Moreover, there exists a map $\xi\colon G\to\mathbb{T}$ with $d(g,h)=\xi_g\xi_hc(\delta(g),\delta(h))\overline{\xi_{gh}}$, a multiplicative character $\eta\colon G\to\mathbb{T}$ and a unitary $w\in\cL_d(H)$ such that

\begin{equation*}
    \Psi(u_g)=\eta_g\xi_{g}wv_{\delta(g)}w^*,\;\;\; \text{for all } g\in G.
\end{equation*}
\end{namedthm*}

In connection with this, Donvil and Vaes were able to establish in \cite[Theorem A]{DV24}, for the first time in the literature, this type of rigidity for certain class of non-property (T) groups $G$ coming from left-right wreath-product groups. More broadly, their results extends to virtual isomorphisms $\Psi$. 

\vspace{1mm}

The results in this paper were obtained in parallel with, and independently of, the work of Donvil and Vaes in \cite{DV24center}. In \cite{DV24center}, the authors establish that a fairly wide class of nontrivial central extensions $\tilde{G}$ of the form $1 \to C \to \tilde{G} \to G \to 1$, where $G$ is a left-right wreath product, are W$^*$-superrigid. A significant portion of their analysis is devoted to proving certain cohomological results, whereas, in our case, these follow more directly from the property (T) assumption. It is worth noting, however, that their results also address W$^*$-superrigidity in the broader context of virtual isomorphisms. While there is a natural overlap in the general approach between this paper and \cite{DV24center}, the technical overlap is relatively small, being limited to the methodology appearing in Section \ref{section5} and our prior work \cite{cfqt23}.

\noindent\textbf{Organization of the article.}  Building on prior results in geometric group theory \cite{Osin07,DGO11,DG18,Sun,cios22} we introduce in Section \ref{section2} a new quotienting technique which yields central extensions by wreath-like product groups, which serve as the primary examples for our results. Along the way, we extend several group-theoretic results from \cite{cios22, cios3}. Section \ref{section3} provides preliminaries on von Neumann algebras that are used in the subsequent sections. In Section \ref{section4} we establish the key rigidity results for twisted group factors of wreath-like product groups which is instrumental in the \emph{fiber analysis} used to prove the main results. These findings also generalize and strengthen earlier work in \cite{cios22, cios22b, cfqt23}. Finally, Section \ref{section5} presents the proof of Theorem \ref{mainresult}, from which \ref{TheoremA}, \ref{theoremB}, and \ref{TheoremC} are derived.

\section{Generalized wreath-like product groups}\label{section2}

\subsection{Generalized wreath-like products}\label{constructionofgroups}

The notion of a wreath-like product of two countable groups $A$, $B$ and a given action $B\curvearrowright I$ on a countable set $I$  was introduced in \cite{cios22}. In this paper, we consider a more general situation.

\begin{defn}\label{Def:GWRP}
Let $B$ be a group and $\Lambda$ be an index set. For every $\lambda  \in \Lambda$, let  $A_\lambda$ be a group and let  $B\curvearrowright I_\lambda$ be an action. We say that a group $W$ is a \emph{generalized wreath-like product} associated with this data if $W$ appears as a short exact sequence

\begin{equation*}
    1\longrightarrow \bigoplus_{\lambda \in \Lambda } \left(\bigoplus_{i\in I_\lambda } A_{\lambda, i}\right) \longrightarrow W \stackrel{\e}\longrightarrow B \longrightarrow 1
\end{equation*}

\noindent such that the following conditions hold:
\begin{enumerate}
\item[(a)]  $A_{\lambda, i}\cong A_\lambda $ for all $\lambda \in \Lambda $ and $i\in I_\lambda$;
\item[(b)] for every $w\in W$,  $\lambda \in \Lambda$, and $i\in I_\lambda$, we have $wA_{\lambda, i}w^{-1}=A_{\lambda, \e(w) i}$. 
\end{enumerate} 

The collection of all generalized wreath-like products satisfying this definition will be denoted by  $\WR(\Al, \{ B\curvearrowright I_\lambda\}_{\lambda\in \Lambda})$. The map $\e$ is called the \emph{canonical homomorphism} associated with the generalized wreath-like product structure of $W$ and the normal subgroup $\oplus_{\lambda \in \Lambda } \left(\oplus_{i\in I_\lambda } A_{\lambda, i}\right )$ is called  \emph{the base of $W$}.

In this paper, we will often deal with the situation when all  $I_\lambda =B$, the action $B\curvearrowright I_\lambda$ being by translation. In this case, we call $W$ a \emph{regular generalized wreath-like product} of groups $\Al$ and $B$ and the category of all such groups will be denoted by $\WR(\Al, B)$.
\end{defn}

When $|\Lambda|=1$, Definition \ref{Def:GWRP} recovers the notion of (ordinary) wreath-like product introduced in \cite{cios22,cios22b}. In \cite{cios22b}, the authors introduced a general quotienting technique for constructing ordinary wreath-like products based on the notion of a Cohen–Lyndon triple (or subgroup) from \cite{Sun}. Since our aim is to expand this construction to collections of subgroups we start by recalling some definitions.   

Let $G$ be a group. For a subset $S\subseteq G$, we denote by $\ll S\rr$ its normal closure in $G$. The \emph{left transversal} of a subgroup $H\leqslant G$ is any subset of $G$ that intersects every left coset of $H$ exactly once. We denote by $LT(H,G)$ the set of all left transversals of $H$ in $G$. 


\begin{defn}\label{CLSun}(\cite[Definition 3.13]{Sun}) Let $G$ be a group, let $\Hl$ be a family of subgroups of $G$, and let $N_{\lambda}\ngroup H_\lambda $ be a normal subgroup, for every $\lambda\in\Lambda$. The triple $(G,\Hl,\{N_{\lambda}\}_{\lambda\in\Lambda})$ has the \textit{Cohen-Lyndon property} if there is a left transversal $T_{\lambda}\in LT(H_{\lambda}\ll N\rr, G)$, for every $\lambda\in\Lambda$, such that 

\begin{equation}\label{Eq:CL}\ll N \rr = \Ast_{\lambda \in \Lambda}\left (\Ast_{t\in T_{\lambda}}tN_\lambda t^{-1}\right),\end{equation}

\noindent where $N = \bigcup_{\lambda\in \Lambda} N_\lambda$. When $N_\lambda=H_\lambda$ for all $\lambda \in \Lambda$ we say the collection $\{H_\lambda\}_{\lambda\in \Lambda}$ \emph{has the Cohen-Lyndon property in $G$}. If $\Lambda =\{\lambda\}$ is a singleton and $N_\lambda =H_\lambda\leqslant G$, we simply say $H_\lambda$ is a \emph{Cohen-Lyndon subgroup} of $G$.
\end{defn}

For further use we recall the main result from \cite{Sun}. 

\begin{thm}[{\cite[Theorem 5.1]{Sun}}]\label{deepclprop} Let $G$ be a group and let $\{H_\lambda\}_{\lambda\in \Lambda} \hookrightarrow^{h} (G,X)$ be a family of hyperbolically embedded subgroups for some $X \subset G$. Then for every $\lambda\in \Lambda$ one can find a finite subset $F_\lambda \subset H_\lambda\setminus \{1\}$ such that for every normal subgroup $N_\lambda \ngroup H_\lambda$ with $N_\lambda \cap F_\lambda=\emptyset$ the triple $(G, \{H_\lambda\}_{\lambda\in \Lambda} ,\{N_\lambda\}_{\lambda\in \Lambda})$ has the Cohen-Lyndon property.    
\end{thm}

The next definition generalizes \cite[Definition 4.7]{cios22b}.

\begin{defn}\label{Def:UWR}
Let $W\in\WR(\Al, \{ B\curvearrowright I_\lambda\}_{\lambda\in \Lambda})$ for some index set $\Lambda$, some groups $A_\lambda$, and some actions $B\curvearrowright I_\lambda$. For any $\lambda\in \Lambda$ and $i\in I_\lambda$, let $P_{\lambda, i}$ denote the preimage of $Stab_B(i)$ in $W$ under the canonical homomorphism $W\to B$. By the definition of a wreath-like product, $P_{\lambda,i}$ normalizes the subgroup $A_{\lambda,i}$ of the base of $W$. In particular, elements of $P_{\lambda, i}$ act as automorphisms of $A_{\lambda, i}$. We say that the generalized wreath-like product $W$ is \emph{untwisted} if all these automorphisms are trivial, i.e., $P_{\lambda,i}$ is contained in the centralizer $C_W(A_{\lambda,i})$  for all $\lambda\in \Lambda$ and all $i\in I_\lambda$. 

The subset of untwisted generalized wreath-like products in $\WR(\Al, \{ B\curvearrowright I_\lambda\}_{\lambda\in \Lambda})$ will be denoted by $\WR_0(\Al, \{ B\curvearrowright I_\lambda\}_{\lambda\in \Lambda})$. 

Note that if the action of $B$ on each $I_\lambda$ is transitive, it suffices to check that $P_{\lambda, i}\le C_W(A_{\lambda,i})$ for all $\lambda \in \Lambda$ and at least one $i\in I_\lambda$ in order to show that $W$ is untwisted.
\end{defn}

The next result generalizing \cite[Theorem 2.7]{cios22b}, follows easily from \cite[Proposition 6.1]{Sun}. 

\begin{thm}\label{Thm:Sun}
Let $G$ be a group, $\Hl$ be a collection of subgroups of $G$, and $N_{\lambda}\ngroup H_\lambda$ be a normal subgroup,  for every $\lambda$, such that $(G,\Hl,\{N_{\lambda}\}_{\lambda\in\Lambda})$ has the Cohen-Lyndon property. Let $N=\ll \bigcup_{\lambda\in \Lambda} N_\lambda \rr$ and $A_\lambda = N_\lambda /[N_\lambda, N_\lambda]$ for all $\lambda\in \Lambda$. Then 

\begin{equation*}
    G/[N,N]\in \WR (\Al, \{G/N\curvearrowright I_{\lambda}\}_{\lambda\in\Lambda}),
\end{equation*}

\noindent where $I_{\lambda}=G/H_{\lambda}N$, the action  $G/N\curvearrowright I_{\lambda}$ is by left multiplication, and the stabilizers of elements of $I_{\lambda}$ are isomorphic to $H_{\lambda}/N_{\lambda}$. Furthermore, if all $H_\lambda$ are abelian, then the wreath-like product is untwisted.
\end{thm}

\begin{proof}
This result is essentially a particular version of \cite[Proposition 6.1]{Sun} rephrased using the terminology of our paper. We keep the proof brief and refer to \cite{Sun} for all unexplained notation and details. 

The assumptions of the theorem allow us to apply \cite[Proposition 6.1]{Sun} to the collection $\Hl$ and normal subgroups $N_\lambda\ngroup H_\lambda $. Observe that, in the notation of \cite{Sun}, we have $Q=G/N$ and $R_\lambda =H_\lambda/N_\lambda$ for all $\lambda \in \Lambda$. From \cite[Proposition 6.1(b)]{Sun} we obtain an isomorphism of $\mathbb Z [G/N]$-modules

\begin{equation*}
    N/[N,N] \cong \bigoplus_{\lambda\in \Lambda} Ind_{H_{\lambda}/N_{\lambda} }^{G/N} A_\lambda.
\end{equation*}

\noindent Using the standard description of induced modules (see, \cite[Chapter III Proposition 5.3]{br}), we obtain that $N/[N,N]$ decomposes as 

\begin{equation*}
    N/[N,N] \cong \bigoplus_{\lambda\in \Lambda} \left( \bigoplus_{i\in G/H_{\lambda}N} A_{\lambda, i}\right),
\end{equation*}

\noindent where $A_{\lambda, i}\cong A_\lambda$ and the action of $G/[N,N]$ on $N/[N,N]$ obeys the rule described in part (b) of Definition \ref{Def:GWRP}. Thus, the short exact sequence $1\to N/[N,N] \to G/[N,N]\to G/N\to 1$ gives us the structure of the desired generalized wreath-like product. 

In the notation of Definition \ref{Def:UWR}, we have $P_{\lambda, i}=H_\lambda[N,N]/[N,N]$ for $i=1H_\lambda N \in I_\lambda$. Note that the action of $P_{\lambda, i}$ on  $A_{\lambda, i}$ is induced by the conjugation action of $H_\lambda$ on $N_\lambda$. If every $H_\lambda$ is abelian, the latter action is trivial and, therefore, the generalized wreath-like product is untwisted.
\end{proof}

The previous two results combined yield large classes of generalized wreath-like product groups.

\begin{cor}\label{sunconsequence} Let $G$ be a group and let $\{H_\lambda\}_{\lambda\in \Lambda} \hookrightarrow^{h} (G,X)$ be a family of hyperbolically embedded subgroups for some $X \subset G$. For every $\lambda\in \Lambda$ there is a finite subset $F_\lambda \subset H_\lambda\setminus \{1\}$ such that all normal subgroups $N_\lambda \ngroup H_\lambda$ with $N_\lambda \cap F_\lambda=\emptyset$ satisfy the following property: If $N=\ll \bigcup_{\lambda\in \Lambda} N_\lambda \rr$ and $A_\lambda = N_\lambda /[N_\lambda, N_\lambda]$, for each $\lambda\in \Lambda$, then 

\begin{equation*}
    G/[N,N]\in \WR (\Al, \{G/N\curvearrowright I_{\lambda}\}_{\lambda\in\Lambda}),
\end{equation*}

\noindent where $I_{\lambda}=G/H_{\lambda}N$, the action  $G/N\curvearrowright I_{\lambda}$ is by left multiplication, and the stabilizers of elements of $I_{\lambda}$ are isomorphic to $H_{\lambda}/N_{\lambda}$. Furthermore, if all $H_\lambda$ are abelian, then the wreath-like product is untwisted.
\end{cor}

For future reference, we now outline the construction of a class of property (T) generalized wreath-like product groups that have a trivial outer automorphism group and trivial abelianization. The construction and its proof are similar to the methods used in \cite[Corollary 3.24]{cios22b}, \cite[Theorem 4.5]{ciossmall}, and \cite[Theorem 2.10]{cios3}. However, for convenience, we include most of the details.

\begin{thm}\label{trivoutgenwreathlike} For every $m\in\N$, there exists $W\in\mathcal{WR}_0(\{A_j\}^m_{j=1},\{B\curvearrowright I_j\}_{j=1}^m)$ satisfying the following properties:

\begin{enumerate}
    \item[(a)] $A_j\cong\Z$, for all $1\leq j\leq m$;
    \item[(b)] $B$ is ICC hyperbolic, the action  $B\curvearrowright I_j$ is transitive with finite stabilizers for all $1\leq j\leq m$, and $\text{\rm Out}(B)=\{1\}$; 
    \item[(c)] $W$ has property (T) and trivial abelianization. 
\end{enumerate}
\end{thm}

\begin{proof}

Let $G$ be an infinite, ICC, hyperbolic group with trivial abelianization and property (T) such that and $\text{Out}(G)=\{1\}$. For instance, such a group exists by \cite[Lemma 2.11]{cios3}. 

Let $a_1,\ldots,a_m\in G$ be pairwise non-conjugate elements of infinite order such that $G$ is hyperbolic relative to the family of cyclic subgroups $\{\langle a_1\rangle,\ldots,\langle a_m\rangle\}$; i.e. $E(a_j)=\langle a_j\rangle$, see for instance  \cite[Theorem 7.11]{bowditch} and \cite[Lemma 3.4]{Ols93}. Let $J_1=\{n_{k,1}:k\in \N\},\ldots, J_m=\{n_{k,m}:k\in\N\}$ be strictly increasing infinite sequences of positive integers such that 

\begin{equation}\label{relprime} n_{k,j} |n_{k+1,j}\text{ and } (n_{k,j},n_{l,j'})=1\text{ for all }k,l\in\N\text{ and }1\leq j\neq j'\leq m.\end{equation} 

\noindent For every $k\in \mathbb N$ consider the normal closure $N_k=\ll (a_1)^{n_{k,1}},\ldots,(a_m)^{n_{k,m}}\rr\lhd G$. By \cite{Osin07,DGO11,Sun,cios22b}, there exists large $l$ such that for every $k\geq l$, the quotient $G/N_k$ is hyperbolic relative to $\{\langle a_1\rangle N_k/N_k,\ldots,\langle a_m\rangle N_k/N_k\}$. Notice that all subgroups $\langle a_j\rangle N_k/N_k$ are cyclic of order $n_{k,j}$ (see also \cite[Theorem 3.28(a)]{cios22b}). Recall that finite orders of elements of every hyperbolic group are uniformly bounded \cite{BG95}. Note that $n_{k,j}\to \infty$, as $k\to \infty$ for every $j$. Using \cite[Theorem 7.19(f)]{DGO11} together with the relative primeness conditions from \eqref{relprime}, we obtain that there exists a large $l$ such that for all $k\geq l$ the only elements of order $n_{k,j}$ in $G/N_k$ are the conjugates of $\langle a_j\rangle N_k/N_k$, for all $1\leq j\leq m$. 

Next, we use \cite[Theorem 3]{DG18} and the same argument from \cite[Lemma 2.10]{cios3} to show there is $l$ large so that $\text{Out}(G/N_k)=\{1\}$ for all $k\geq l$. For convenience we reproduce the argument here. 

Assume that there exists an infinite set $X\subset\N$ such that, for every $x\in X$, there is a non-inner automorphism $\alpha_x\in \text{Aut}(G/N_x)$. Let $\varepsilon_x\colon G\to G/N_x$ denote the natural homomorphism. From the prior paragraph, each $\alpha_x$ will map the cyclic subgroup $\langle a_j\rangle N_x/N_x$ onto its conjugates, for $1\leq j\leq m$. Thus, by \cite[Theorem 3]{DG18} (see also \cite[Theorem 2.9]{cios3}) there exists $\alpha\in\text{Aut}(G)$, an infinite subset $Y\subset \N$, and inner automorphisms $\iota_x\in \text{Inn}(G/N_x)$ such that the diagram

\[\begin{tikzcd}
	G &&& G \\
	\\
	{G/N_x} &&& {G/N_x}
	\arrow["\alpha", from=1-1, to=1-4]
	\arrow["\varepsilon_x", from=1-1, to=3-1]
	\arrow["\varepsilon_x", from=1-4, to=3-4]
	\arrow["{\iota_x\circ\alpha_x}", from=3-1, to=3-4]
\end{tikzcd}\]

\noindent is commutative for all $x\in Y$. Since $\text{Out}(G)=\{1\}$, $\alpha$ is an inner automorphism of $G$. Consequently,  $\alpha_x\in\text{Inn}(G/N_x)$ for all $x\in Y$, which contradicts the assumption about the choice of $\alpha_x$. This implies that one can find $l$ large enough such that $ \text{Out}(G/N_k)=\{1\}$ for all $k\geq l$.

Using Theorem \ref{deepclprop} there is $l$ large such that for all $k\geq l$ the triple $(G,\{\langle a_j\rangle\}_{j=1}^m,\{\langle (a_j)^{n_{k,j}}\rangle\}_{j=1}^m)$ satisfies the Cohen-Lyndon property as in Definition \ref{CLSun}. Thus, for each $1\leq j\leq m$, we can find a left transversal $I_{k,j}\in LT(\langle a_j\rangle N_k,G)$ satisfying

\begin{equation*}
    N_k=\underset{1\leq j\leq m}{\ast}\left(\underset{t\in I_{k,j}}{\ast} t\langle (a_j)^{n_{k,j}}\rangle t^{-1} \right),\:\text{ for all }k\geq l.
\end{equation*}

\noindent Now consider the commutator subgroup $[N_k,N_k]\lhd G$. By Theorem \ref{Thm:Sun}, for every $k\geq l$, the quotient $W_k=G/[N_k,N_k]$ satisfies $W_k\in\mathcal{WR}_0(\{A_{j}\}^m_{j=1},\{G/N_k\curvearrowright I_{k,j}\}_{j=1}^m)$ where $A_{j} \cong \mathbb Z$ for all $1\leq j\leq m$, and the action $G/N_k\curvearrowright I_{k,j}$ is transitive with finite stabilizers. 

Finally, note that $W_k$ has property (T) and trivial abelianization being a quotient of $G$.
\end{proof}

In the remainder of this subsection, we will focus on the construction of regular generalized wreath-like product groups.

Recall that a hyperbolic group is called \emph{elementary} if it is virtually cyclic. We will also need the following simplification of \cite[Proposition 4.21]{cios22b}.

\begin{prop}\label{Prop:CLS}
Let $G$ be a non-elementary hyperbolic group. For every $n\in \NN$, there exists a Cohen--Lyndon subgroup $H\leqslant G$ such that $H$ is free of rank $n$, and $G/\ll H\rr$ is non-elementary hyperbolic. Moreover, if $G$ is torsion-free, then so is $G/\ll H\rr$.
\end{prop}

Further, we record a simple observation.

\begin{lem}\label{Lem:CLC}
Suppose that $H$ is a Cohen--Lyndon subgroup of a group $G$ and $H=\Ast_{\lambda\in \Lambda} H_\lambda$. Then the collection $\Hl$ has the Cohen--Lyndon property in $G$.
\end{lem}
\begin{proof}
By the definition of a Cohen--Lyndon subgroup, there exists a left transversal $T\in LT(\ll H\rr, G)$ such that 

\begin{equation*}
    \ll H\rr = \Ast_{t\in T} tHt^{-1} = \Ast_{t\in T} \left( \Ast_{\lambda\in \Lambda} tH_\lambda t^{-1}\right).
\end{equation*}

\noindent Since $\ll \bigcup_{\lambda\in \Lambda} H_\lambda \rr=\ll H\rr$, we obtain (\ref{Eq:CL}).
\end{proof}

The next lemma generalizes \cite[Lemma 2.12]{cios22}.

\begin{lem}\label{Lem:A/N}
Let $\Lambda $ be a set, $\Al$ a collection of arbitrary groups and $W\in \WR (\Al, B)$. We identify each $A_\lambda $ with the subgroup $A_{\lambda,1}$ of the base of $W$. For any collection of normal subgroups $N_\lambda \ngroup A_\lambda$, we have 

\begin{equation*}
    W\left/\ll \bigcup_{\lambda\in \Lambda} N_\lambda \rr \right.\in \WR (\{ A_\lambda/N_\lambda\}_{\lambda\in \Lambda}, B).
\end{equation*}
\end{lem}

\begin{proof}
For every $b\in B$ and $\lambda \in \Lambda$, we define $N_{\lambda, b}=uN_\lambda u^{-1}$, where $u$ is an element of $W$ such that

\begin{equation}\label{eq:eu=b}
\e(u)=b.
\end{equation}

\noindent Note that the subgroup $N_{\lambda, b}$ is independent of the choice of a particular element $u\in W$ satisfying (\ref{eq:eu=b}). Indeed, if $v\in W$ is another element such that $\e(v)=b$, then $u^{-1}v$ belongs to the base of $W$. Obviously, $N_\lambda $ is normal in the base. Therefore, $(u^{-1}v)N_\lambda (u^{-1}v)^{-1}=N_\lambda$, which implies $uN_\lambda u^{-1}=vN_\lambda v^{-1}$.

Clearly, $N_{\lambda,b}\lhd A_{\lambda, b}$ for all $b\in B$ and $\lambda \in \Lambda$. Further, let $M=\ll \bigcup_{\lambda\in \Lambda} N_\lambda \rr$. It is easy to see that $M=\bigoplus_{\lambda\in \Lambda}\left( \bigoplus_{b\in B} N_{\lambda, b}\right)$.  
Therefore, the group $W/M$ splits as

\begin{equation*}
    1\longrightarrow \bigoplus_{\lambda\in \Lambda}\left( \bigoplus_{b\in B} A_{\lambda, b}/N_{\lambda, b}\right) \longrightarrow  W/M\stackrel{\delta}\longrightarrow B\longrightarrow 1,
\end{equation*}

\noindent where $\delta $ is induced by the canonical homomorphism $W\to B$. It remains to note that we have $w(A_{\lambda, b}/N_{\lambda, b}) w^{-1}=A_{\lambda, \delta(w) b}/N_{\lambda, \delta(w) b}$ for all $w\in W/M$, all $\lambda\in \Lambda$, and all $b\in B$.
\end{proof}

The following result, which generalizes \cite[Theorem 4.24]{cios22b}, provides many examples of regular generalized wreath-like product groups.

\begin{thm}\label{Thm:Hyp-quot}
For any finite set $\Lambda $, any finite collection of finitely generated abelian groups $\Al$, and any non-elementary hyperbolic group $G$, there exists a quotient group $W$ of $G$ such that $W\in \WR(\Al, B)$, where $B$ is non-elementary hyperbolic. In addition, if $G$ is torsion-free, then so is $B$. Also, if $G$ has property (T) then so does $W$.
\end{thm}

\begin{proof}
Let $G$ be a non-elementary hyperbolic group. For each $\lambda\in \Lambda$, we represent $A_\lambda $ as a quotient group $Z_\lambda/N_\lambda$, where $Z_\lambda \cong \ZZ^{r_\lambda}$.  Let $H$ be the Cohen--Lyndon subgroup of $G$ given by  Proposition \ref{Prop:CLS} such that $H$ is free of rank $n=\sum_{\lambda\in \Lambda} r_\lambda$. We have $H=\Ast_{\lambda\in \Lambda} H_{\lambda}$, where each $H_\lambda$ is a free group of rank $r_\lambda$. By Lemma \ref{Lem:CLC}, $\Hl$ satisfies the Cohen--Lyndon property in $G$. 

Letting $N=\ll H\rr = \ll \bigcup_{\lambda\in \Lambda}H_\lambda\rr $ and applying Theorem \ref{Thm:Sun}, we obtain that 

\begin{equation*}
    G/[N,N]\in \WR(\{Z_\lambda\}_{\lambda\in \Lambda}, G/N).
\end{equation*}

\noindent Now, applying Lemma \ref{Lem:A/N}, we obtain a quotient group of $G/[N,N]$ that belongs to the class $\WR(\{Z_\lambda/N_\lambda\}_{\lambda\in \Lambda}, G/N)=\WR(\Al, G/N)$. It remains to note that $G/N$ is torsion-free whenever so is $G$ by Proposition \ref{Prop:CLS}.
\end{proof}

\subsection{Central extensions of generalized wreath-like product groups} In the first part of this subsection, we introduce a canonical quotienting technique for wreath-like product groups, which yields large classes of nonsplit central extensions, many with property (T) (see \ref{TheoremD}). In the second part, we construct central extensions that additionally have a trivial outer automorphism group (see \ref{centralextwithOut=1}). Together, these constructions will serve as the main source of examples for the key results of our paper (e.g., \ref{TheoremA} and \ref{TheoremC}).

\begin{thm}\label{Thm:CenExt}
Let $\Lambda$ be an index set together with a finite subset $\Upsilon\subset \Lambda$, $B$ and $\{A_\lambda\}_{\lambda\in \Lambda} $ any groups,  also $\{B\curvearrowright I_\lambda\}_{\lambda\in\Lambda}$ a collection of transitive actions with finite stabilizers.  Let also $U\in \WR_0(\{ A_\lambda \}_\lambda, \{B\curvearrowright I_\lambda\}_{\lambda\in\Lambda})$. Suppose that for every $\lambda \in \Upsilon$, one of the following conditions holds:

\begin{enumerate}
    \item[(a)] $A_\lambda \cong \ZZ$; 
    \item[(b)] $A_\lambda $ is abelian and the action $B\curvearrowright I_\lambda$ is regular.
\end{enumerate} 

\noindent Then $U$  admits an epimorphism onto a central extension $Q$ of the form 

\begin{equation}\label{extension}
    1\to \bigoplus_{\lambda\in \Upsilon} A_\lambda  \to Q\to W\to 1,
\end{equation}

\noindent where $W\in \WR_0(\{A_\lambda \}_{\lambda\in \Lambda\setminus\Upsilon},\{B\curvearrowright I_\lambda\}_{\lambda \in \Lambda\setminus\Upsilon})$. Moreover, the associated 2-cocycle of the extension \eqref{extension} may be chosen to be trivial on $\oplus_{\lambda\in\Lambda\setminus\Upsilon}A^{(I_{\lambda})}$.
\end{thm}

\begin{proof}
Let $U\in \WR_0(\{ A_\lambda \}_\lambda, \{B\curvearrowright I_\lambda\}_{\lambda\in\Lambda})$. We fix a section $\sigma $ of the canonical homomorphism $\e\colon U\to B$. That is, $\sigma$ is any map $B\to U$ such that $\e\circ \sigma \equiv id_B$. Below we think of elements of $A_\lambda^{(I_\lambda)}\leqslant U$ as finitely supported functions $I_\lambda\to A_\lambda$ and define a map $\pi\colon \oplus_{\lambda \in \Upsilon} A_\lambda^{(I_\lambda)}\to \oplus_{\lambda \in \Upsilon}A_\lambda$ by the rule

\begin{equation*}
    \pi(\oplus_\lambda f_\lambda) =\prod_{\lambda\in \Upsilon,\: b\in B} \big(\sigma(b)^{-1} f_\lambda\sigma (b)\big) (i_\lambda)
\end{equation*}

\noindent for all $f_\lambda\in A_\lambda^{(I_\lambda)}$, where $i_\lambda\in I_\lambda$ is a fixed element. Note that $\pi(\oplus_\lambda f_{\lambda})$ is well-defined. Indeed, conjugation by $\sigma(b)$ defines an isomorphism between $A_{\lambda,b\cdot j }$ and $A_{\lambda,j}=\sigma(b)^{-1} A_{\lambda, b\cdot j}\sigma(b)$, for each $\lambda$. This isomorphism sends $f_{\lambda}(b\cdot i_\lambda)$ to $\sigma(b)^{-1} f_{\lambda}\sigma (b)(i_\lambda)$. Thus, $\sigma(b)^{-1} f_{\lambda}\sigma (b)(i_\lambda)\ne 0$ if and only if $f_{\lambda}(b\cdot i_\lambda)\ne 0$. Hence, since $\Upsilon$ is finite and the action $B\curvearrowright I_\lambda$ has finite stabilizers, the product $\prod_{\lambda\in\Upsilon, b\in B} \big(\sigma(b)^{-1} f_\lambda\sigma (b) (i_\lambda)\big)$ has only finitely many non-trivial terms. 

In addition, it is easy to see that $\pi$ is a homomorphism. Consider the following subgroup 

\begin{equation*}
    N=\Ker (\pi)=\{\oplus_{\lambda} f_{\lambda}\in \oplus_{\lambda \in \Upsilon} A_\lambda^{(I_\lambda)}\mid \pi(\oplus_{\lambda}f_{\lambda})=1\}.
\end{equation*}

\noindent We think of $N$ as a subgroup of $U$. For any $u\in U$ and any $b\in B$, we have $u\sigma (b)=a\sigma(ub)$ for some $a\in \Ker(\e)$. Since $\Ker(\e)\:\cap\: \oplus_{\lambda \in \Upsilon} A_\lambda^{(I_\lambda)}$ is abelian, we have $\sigma(b)^{-1} u^{-1}f_{\lambda}u\sigma (b)=\sigma(\varepsilon(u)b)^{-1}f_{\lambda}\sigma(\varepsilon(u)b)$ for all $f_{\lambda}\in A_\lambda^{(I_\lambda)}$, $\lambda\in\Upsilon$. Consequently,

\begin{align*}\label{rcl}
    \pi (u^{-1}\oplus_{\lambda\in\Upsilon}f_{\lambda}u) & = \prod_{\lambda\in\Upsilon,\: b\in B} \big(\sigma(b)^{-1} u^{-1}f_{\lambda}u\sigma (b) (i_\lambda)\big) \\
    &=\prod_{\lambda\in\Upsilon,\: b\in B} \big(\sigma(\varepsilon(u)b)^{-1} f_{\lambda}\sigma (\varepsilon(u)b) (i_\lambda)\big)\\
    &=\prod_{\lambda\in\Upsilon,\: b^\prime\in B} \big(\sigma(b^\prime )^{-1} f_{\lambda}\sigma (b^\prime) (i_\lambda)\big) = \pi (\oplus_{\lambda\in\Upsilon}f_{\lambda}),
\end{align*}

\noindent where $b^\prime =\varepsilon(u)b$. Thus, $N\lhd U$. It follows that $\pi$ extends to a map $\widehat\pi\colon \oplus_{\lambda\in \Upsilon}A_\lambda^{(I_\lambda)}\to U/N$.  

It is straightforward to verify that the image of $\oplus_{\lambda\in\Upsilon}A_\lambda^{(I_\lambda)}$ under $\widehat\pi$ is a central subgroup of $U/N$ isomorphic to $\oplus_{\lambda\in\Upsilon}A_\lambda^\prime$, where $A_\lambda^\prime$ is a subgroup of $A_\lambda $ for every $\lambda$. Given $\lambda\in \Upsilon$, consider a function $f_\lambda \colon I_\lambda \to A_\lambda=\mathbb Z$ such that $f_\lambda(i_\lambda)\ne 0$ and  $f_\lambda(i)=0$ for all $i\ne i_\lambda$. Using the fact that the generalized wreath product $U$ is untwisted, we obtain 

\begin{equation}\label{Eq:f}
\widehat\pi (f_\lambda)=\pi(f_\lambda)= f_\lambda(i_\lambda)^{|Stab_B(i_\lambda)|}
\end{equation}

\noindent (we use the multiplicative notation for the abelian groups $A_\lambda$ here.)
If condition (a) from the assumption of the theorem holds, equality (\ref{Eq:f}) implies that each $A_\lambda^\prime$ is non-trivial; therefore, $A_\lambda ^\prime \cong A_\lambda$ since every non-trivial subgroup of $\mathbb Z$ is isomorphic to $\mathbb Z$. If condition (b) holds, we have $\widehat\pi (f_\lambda)=f_\lambda$ by (\ref{Eq:f}), which obviously implies $A^\prime_\lambda =A_\lambda$. In either case, we have

$$\widehat\pi \left(\oplus_{\lambda\in\Upsilon}A_\lambda^{(I_\lambda)}\right) \cong \bigoplus_{\lambda\in \Upsilon} A_\lambda.$$ 

\noindent Using the isomorphism theorems, we obtain 

$$(U/N)\Big/\pi \left(\oplus_{\lambda\in\Upsilon}A_\lambda^{(I_\lambda)}\right) \in \WR_0(\{A_\lambda \}_{\lambda\in \Lambda\setminus\Upsilon},\{B\curvearrowright I_\lambda\}_{\lambda \in \Lambda\setminus\Upsilon}).$$

Finally, the extension splits over $\oplus_{\lambda\in\Lambda\setminus\Upsilon}A^{(I_{\lambda})}$ since $U\to Q$ is injective on $\oplus_{\lambda\in\Lambda\setminus\Upsilon}A^{(I_{\lambda})}$, and hence, the associated 2-cocycle may be chosen to vanish on this subgroup.
\end{proof}

\begin{proof}[\textbf{Proof of \ref{TheoremD}}]
Let $L$ be a uniform lattice in $Sp(n,1)$. By Selberg’s lemma \cite{selberg}, $L$ contains a torsion-free subgroup $K$ of finite index. Being a finite index subgroup of $L$, $K$ is also a uniform lattice in $Sp(n,1)$. Therefore, $K$ is hyperbolic and has property (T). By \cite[Corollary 3.24]{cios22b}, every non-cyclic, torsion-free hyperbolic group has a torsion-free, non-cyclic, hyperbolic quotient with trivial abelianization. Thus, $K$ has a torsion-free, non-cyclic, hyperbolic quotient group $G$ with trivial abelianization.

Let $U\in \WR(\{A,C\}, B)$ be the quotient group of $G$ provided by Theorem \ref{Thm:Hyp-quot}; in particular, $B$ is non-elementary, torsion-free, and hyperbolic. Applying Theorem \ref{Thm:CenExt}, we obtain an epimorphism from $U$ to a central extension $Q$ of the form $1\to C\to Q\to W\to 1$, where $W\in \WR(A,B)$ and $Z(Q)=C$. This proves (a). It remains to note that $Q$ satisfies (b) from the moreover part of the previous theorem, and it satisfies (c) being a quotient of $G$.
\end{proof}

In the following, we will focus on constructing central extensions with a trivial outer automorphism group. To do this, we first need the following elementary result concerning automorphisms of groups with trivial abelianization.

\begin{prop}\label{outcontrol1}
For any group $G$ with trivial abelianization, $\Out(G)$ is isomorphic to a subgroup of $\Out(G/Z(G))$. In particular, if $\Out(G/Z(G))$ is trivial, then so is $\Out(G)$.
\end{prop}

\begin{proof} We define a map $\e\colon \Out(G)\to \Out(G/Z(G))$ as follows. Let $\alpha \in \Aut(G)$. Since $\alpha (Z(G))= Z(G)$, the automorphism $\alpha$ induces an automorphism $\widehat\alpha \in \Aut(G/Z(G))$ by the rule $\widehat\alpha (gZ(G))=\alpha (g)Z(G)$ for all $g\in G$. Further, we let 

$$\e(\alpha \Inn(G))=\widehat \alpha \Inn(G/Z(G)).$$ 

\noindent It is straightforward to verify that $\e$ is a well-defined homomorphism.

We want to show that $\e$ is injective. To this end, assume that $\e(\alpha \Inn(G))= \Inn(G/Z(G))$. It suffices to show that $\alpha \Inn(G)=\Inn(G)$. By definition, we have $\widehat \alpha\in \Inn (G/Z(G))$. Replacing $\alpha$ with another automorphism in $\alpha \Inn (G)$ if necessary, we can assume that $\widehat\alpha$ is the identity map. That is, for every $g\in G$, we have 

 \begin{equation}\label{Eq:a(g)}
     \alpha (g)= gz_g
 \end{equation}
 
\noindent for some $z_g\in Z(G)$. Since $G=[G,G]$, for every $g\in G$, there exist $n\in \mathbb N$ and $a_1, b_1, \ldots, a_n, b_n\in G$ such that $g=[a_1, b_1]\cdots[a_n, b_n]$. Using (\ref{Eq:a(g)}), we obtain
 
$$
 \alpha (g)= \alpha([a_1, b_1]\cdots[a_n, b_n])= [a_1z_{a_1}, b_1z_{b_1}]\cdots[a_nz_{a_n}, b_nz_{b_n}]= [a_1, b_1]\cdots[a_n, b_n] =g.
$$
 
\noindent Thus, $\alpha \Inn(G)=\Inn(G)$ and $\e$ is injective. 
\end{proof}

For further use, we also note the following consequence of the previous result.

\begin{cor}\label{trivout} Let $W \in {\mathcal WR}(A, B\curvearrowright I)$ be a property (T) group where $A$ is abelian and $B\curvearrowright I$ has infinite orbits. Consider a central extension $1\to Z(G)\to G\to W\to 1$, where $G$ has property (T) and trivial abelianization. Then we have

\begin{equation*}
    |{\rm Out}(G)|\le |{\rm Out}(W)|\le |{\rm Out}(B)|.
\end{equation*}

\noindent In particular, if ${\rm Out}(B)=\{1\}$ or ${\rm Out}(W)=\{1\}$ then ${\rm Out}(G)=\{1\}$.
\end{cor}

\begin{proof} The first inequality follows from Proposition \ref{outcontrol1}. The second inequality follows using verbatim the arguments from \cite[Lemma 4.26, Proposition 4.27 and Corollary 6.9]{cios22b}. We leave the details to the reader.    
\end{proof}

With these preparations at hand we are now ready to prove the other main result of this section.

\vspace{1mm}

\begin{proof}[\textbf{Proof of \ref{centralextwithOut=1}}] 
Let $U\in\mathcal{WR}_0(\{A_j\}^{n+1}_{j=1},\{B\curvearrowright I_j\}_{j=1}^{n+1})$ be a property (T) group satisfying all the conditions from Theorem \ref{trivoutgenwreathlike} for $m=n+1$. Applying Theorem \ref{Thm:CenExt} to $U$ and the sets $\Upsilon=\{1,\ldots,n\}
\subset \{1,\ldots,n+1\}= \Lambda$, one obtains a central extension 

$$1\to \mathbb Z^n\to Q\to W\to 1$$ 

\noindent that splits over the subgroup $A_{n+1}^{(I_{n+1})}< W$. Since $W \in \WR( A_{n+1}, B\curvearrowright I_{n+1})$, where $A_{n+1}\cong \mathbb Z$, $B$ is ICC hyperbolic and the action $B \curvearrowright I_{n+1}$ is transitive with finite stabilizers, $W$ is ICC by \cite[Lemma 4.11(b)]{cios22}. Hence, we have $Z(W)=\{ 1\}$, and therefore $Z(Q)\cong \mathbb Z^n$. 
Since $U$ has property (T) and trivial abelianization, then so is its quotient $Q$.  Finally, since ${\rm Out}(B)=\{1\}$, by Corollary \ref{trivout} we have ${\rm Out}(Q)=\{1\}$. \end{proof}

We end the section with the following independent result that will be used in the von Neumann algebraic part of the paper.

\begin{lem}\label{finiteindex[W,W]} Let $1 \ra Z(G)\ra G\overset{\varepsilon}{\ra} W\ra 1$ be a central extension, where $G$ has property (T). Let $\sigma\colon  W\ra G $ be a section with $\varepsilon\circ \sigma =id_W$ and denote by $c\colon  W \times W \ra Z(G)$ the natural $2$-cocycle associated with this extension, i.e. $\sigma(g)\sigma(h)= c(g,h)\sigma(gh)$ for all $g,h\in W$. Let $W_0\leqslant W$ be a finite index subgroup and consider the subgroup of $Z(G)$ generated by $D=\langle c(g,h)\,:\, g,h\in W_0\rangle\leqslant Z(G)$. Then the subgroup generated by $\langle D, \sigma(W_0) \rangle \leqslant G$ has finite index.
\end{lem}

\begin{proof} Since $W_0\leqslant W$ has finite index then so does $ \varepsilon^{-1}(W_0)\leqslant G$. In particular, $\varepsilon^{-1}(W_0)$ has property (T). Denote by $Q=\langle D, \sigma(W_0) \rangle$ and notice that $\varepsilon^{-1}(W_0)/Q\cong Z(G)/D$. Thus $\varepsilon^{-1}(W_0)/Q$ is abelian and has property (T), which means it is finite. Therefore, $Q\leqslant \varepsilon^{-1}(W_0)$ has finite index and hence $Q\leqslant G$ has finite index as well.
\end{proof}

\section{Preliminaries on von Neumann algebras}\label{section3}

\subsection{Intertwining techniques} In \cite[Theorem 2.1, Corollary 2.3]{popa06} Popa introduced the following  weak-mixing like criterion for deciding existence of an intertwining  between von Neumann subalgebras of a given tracial von Neumann algebra. 

\begin{thm} Let $(\cM,\tau)$ be a tracial von Neumann algebra with $\mathcal{P},\mathcal{Q}\subseteq \cM$ von Neumann subalgebras where the inclusions are not necessarily unital. Then, the following are equivalent:

\begin{enumerate}
    \item[(1)] There exists non-zero projections $p\in \mathcal{P}$, $q\in\mathcal{Q}$, a $*$-homomorphism $\Theta\colon p\mathcal{P}p\to q\mathcal{Q} q$ and a non-zero partial isometry $v\in q\cM p$ such that $\Theta(x)v=vx$, for all $x\in p\mathcal{P}p$.
    \item[(2)] There exists no net $u_n\in\mathcal{P}$ satisfying $\|\mathbb{E}_{\mathcal{Q}}(x^*u_ny)\|_2\to 0$, for all $x,y\in\cM$. 
\end{enumerate}
\end{thm}

If (1) or (2) hold, we write $\mathcal{P}\prec_{\cM}\mathcal{Q}$ and say that \textit{a corner of $\mathcal{P}$ embeds into $\mathcal{Q}$ inside $\cM$.} If $\mathcal{P}p'\prec_{\cM}\mathcal{Q}$ for any nonzero projection $p'\in\mathcal{P}'\cap \cM$, we write $\mathcal{P}\prec_{\cM}^{\rm s}\mathcal{Q}$.

\subsection{Twisted group von Neumann algebras}\label{twistedvNdef} A \emph{cocycle action} of a group $G$ on a tracial von Neumann algebra $(\cM,\tau)$ is a pair $(\alpha,c)$ consisting of two maps $\alpha\colon G\to {\text{Aut}}(\cM)$ and $c\colon G\times G\to\mathscr{U}(\cM)$ which satisfy 

\begin{enumerate}
    \item[(1)] $\alpha_g\alpha_h={\text {Ad}}(c(g,h))\alpha_{gh}$ for every $g,h\in G$,
    \item[(2)] $c(g,h)c(gh,k)=\alpha_g(c(h,k))c(g,hk)$ for every $g,h,k\in G$, and
    \item[(3)] $c(g,e)=c(e,g)=1$ for every $g\in G$.
\end{enumerate} 

\noindent Let $G \curvearrowright^{\alpha,c} (\mathcal M,\tau)$ be a cocycle action.  The \emph{cocycle crossed product} von Neumann algebra, denoted by $\cM\rtimes_{\alpha,c}G$, is a tracial von Neumann algebra which is generated by a copy of $\cM$ and unitary elements $\{u_g\}_{g\in G}$ such that $u_gxu_g^*=\alpha_g(x)$, $u_gu_h=c(g,h)u_{gh}$, and $\tau(xu_g)=\tau(x)\delta_{g,e}$ for every $g,h\in G$ and $x\in \cM$.

Let $\mathcal M=\mathbb C$, let $\alpha\colon G\to\C$ be the trivial action and let $c\colon G\times G\to \mathbb{T}$ be a 2-cocycle, i.e.\ $c(g,h)c(gh,k)=c(h,k)c(g,hk)$, for all $g,h,k\in G$. Then, $G\curvearrowright^{\alpha,c}\C$ is a cocycle action and the corresponding cocycle crossed product $\C\rtimes_{\alpha,c}G$ is called the \emph{twisted group von Neumann algebra} $\cL_c(G)$. This is a tracial von Neumann algebra generated by the group unitaries $\{u_g\}_{g\in G}$, satisfying $u_gu_h=c(g,h)u_{gh}$ and $\tau(u_g)=\delta_{g,e}$ for all $g,h\in G$.

\begin{rem}\label{cocyclecohtonormal} For a 2-cocycle $c\colon G\times G\to\mathbb{T}$ we can assume $c(g,g^{-1})=1$ for all $g\in G$. To see this, define $f\colon G\to \mathbb{T}$ by $f(g)=c(g,g^{-1})^{-1/2}=f(g^{-1})$. In this case, $c'(g,h):=c(g,h)f(g)f(h)f(gh)^{-1}$ is cohomologous to $c$ and $c'(g,g^{-1})=1$ for all $g\in G$. In particular, this gives that $c'(g,h)=\overline{c'(h^{-1},g^{-1})}$ for all $g,h\in G$.
\end{rem}

\subsection{Direct integral decompositions} We use the theory of direct integrals of von Neumann algebras with separable predual, referencing \cite[Section IV.8]{takesaki} and \cite[Section 6]{cfqt23} for the requisite background information. 

For $\cM$ a von Neumann algebra with separable predual, and $(X,\mu)$ a standard probability space with $L^{\infty}(X,\mu)\subseteq\mathscr{Z}(\cM)$, we have a direct integral decomposition $\cM=\int_X^{\oplus}\cM_xd\mu(x)$. When the decomposition is taken over the center, the fibers $\cM_x$ are factors almost everywhere.

We will apply direct integral decompositions in the case of group von Neumann algebras $\cL(G)$ coming from central extensions $1\ra Z\ra G \overset{\e}{\ra} G_0\ra 1$ where $Z\leqslant Z(G)$. We denote by $c\colon  G_0\times G_0\ra Z$ the $2$-cocycle satisfying $\sigma(g)\sigma(h)=c(g,h)\sigma(gh)$ for all $g,h\in G_0$ for $\sigma\colon  G_0\ra G$ a section with $\e\circ\sigma=id_{G_0}$. Let $(\hat{Z},\mu)$ be the dual of $Z$ with its Haar measure and identify $\cL(Z)=L^{\infty}(\hat{Z},\mu)$. Following \cite[Proposition 6.6]{cfqt23} we can find a measurable field of cocycles $c_x\colon G_0\times G_0\to\mathbb{T}$ ``fibering'' the von Neumann algebraic $2$-cocycle induced by $c$ in such a way that

\begin{equation*}
    \cL(G)=\int_{\hat{Z}}^{\oplus}\cL_{c_x}(G_0)d\mu(x).
\end{equation*}

\section{Proof of \ref{strongsuperrig}}\label{section4}

\noindent In this section, we present our rigidity result for twisted group factor von Neumann algebras, which builds upon and extends the techniques developed in \cite[Theorem 1.3]{cios22} and \cite[Theorem C]{cfqt23}. Before proceeding, we review some notation and key preliminary results.

Let $G$ and $H$ be groups, and let $d\colon G\times G\to\mathbb{T}$ and $c\colon  H \times H \to \mathbb T$ be $2$-cocycles. From Remark \ref{cocyclecohtonormal}, assume $d$ satisfies $d(g,h)=\overline{d(h^{-1},g^{-1})}$ for all $g,h\in G$. Suppose that $\cL_d(G)^t=\cL_{c}(H)$ for some $t>0$. Next we recall the notion of triple commultiplication from \cite{ioa10} (see also \cite{DV24,cfqt23}). This is the $\ast$-embedding $\Delta_0\colon \cL_{c}(H)\to \cL_{c}(H)\:\overline{\otimes}\:\cL_{c}(H)^{\rm op}\:\overline{\otimes}\:\cL_{c}(H)$ given by 

\begin{equation}\label{3commul}
    \Delta_0(v_h)= \overline{c(h^{-1},h)}\: v_h\otimes \overline v_{h^{-1}}\otimes v_h\text{ for }h\in H.
\end{equation} 

\begin{rem} Let $n$ be the smallest integer such that $n\geq t$. Denote by $\cM=\cL_d(G)$ and define $\tilde{\cM}:=\cM\:\overline{\otimes}\:\cM^{\text{op}}\:\overline{\otimes}\:\cM\:\overline{\otimes}\:\mathbb{M}_n(\C)\:\overline{\otimes}\:\mathbb{M}_n(\C)^{\text{op}}$. Then, the map $\Delta_0$ can be amplified to a unital $*$-homomorphism $\Delta\colon \cM\to p\tilde{\cM}p$, where $p\in\tilde{\cM}$ is a projection with $(\tau\otimes\tau^{\text{op}}\otimes\tau\otimes \text{Tr}\otimes\text{Tr}^{\text{op}})(p)=t^2$. We make the construction of $\Delta$ explicit as in \cite[Remark 4.9]{cios22}. 

Assume that $t\in\N$ so that $t=n$, $p=1$ and $\cL_c(H)=\cM\:\overline{\otimes}\:\mathbb{M}_n(\C)$. Let $\psi\colon \tilde{\cM}\:\overline{\otimes}\:\mathbb{M}_n(\C)\to \cM\:\overline{\otimes}\:\mathbb{M}_n(\C)\:\overline{\otimes}\:\cM^{\text{op}}\:\overline{\otimes}\:\mathbb{M}_n(\C)^{\text{op}}\:\overline{\otimes}\:\cM\:\overline{\otimes}\:\mathbb{M}_n(\C)$ be the $*$-isomorphism given by 

\begin{equation*}
    \psi(a\otimes b\otimes c\otimes d\otimes e\otimes f)=a\otimes d\otimes b\otimes e\otimes c\otimes f.
\end{equation*}

\noindent Let $U$ be a unitary in $\cM\:\overline{\otimes}\:\mathbb{M}_n(\C)\:\overline{\otimes}\:\cM^{\text{op}}\:\overline{\otimes}\:\mathbb{M}_n(\C)^{\text{op}}\:\overline{\otimes}\:\cM\:\overline{\otimes}\:\mathbb{M}_n(\C)$ such that $\Delta_0(1_{\cM}\otimes x)=U\psi(1_{\tilde{\cM}}\otimes x)U^*$, for every $x\in\mathbb{M}_n(\C)$. Then,

\begin{equation}\label{Delta0}
    \psi^{-1}\circ\text{Ad}(U^*)\circ\Delta_0\colon \cM\:\overline{\otimes}\:\mathbb{M}_n(\C)\to\tilde{\cM}\:\overline{\otimes}\:\mathbb{M}_n(\C)
\end{equation}

\noindent is a unital $*$-homomorphism leaving $1\:\overline{\otimes}\:\mathbb{M}_n(\C)$ fixed, so it can be written as $\Delta\otimes\text{Id}_n$ where $\Delta\colon \cM\to\tilde{\cM}$. Thus, $\Delta_0=\text{Ad}(U)\circ\psi\circ(\Delta\otimes \text{Id}_n)$.
\end{rem}

In the proof of \ref{strongsuperrig}, we will combine \eqref{Delta0} with some properties of the triple commultiplication: 

\begin{equation}\label{flipproperty}
\begin{split}&(\Delta_0\otimes\text{id}\otimes\text{id})\circ \Delta_0=(\text{id}\otimes\text{id}\otimes\Delta_0)\circ \Delta_0\text{ and } \\
&F_{1,2}\circ(\Delta_0\otimes\text{id}\otimes\text{id})\circ \Delta_0=F_{3,4}\circ(\text{id}\otimes\text{id}\otimes\Delta_0)\circ \Delta_0.
\end{split}\end{equation}

\noindent Here $F$ is the star-flip map $F\colon  \mathcal L_c(H)\:\overline{\otimes}\:\mathcal L_c(H)^{\rm op} \rightarrow \cL_c(H)\:\overline{\otimes}\: \cL_c(H)^{\rm op}$ given by $F(x\otimes \overline{y})=y^*\otimes\overline{x^*}$, 
$F_{1,2}=F\otimes \text{id}\otimes \text{id}\otimes \text{id}$ and $F_{3,4}=\text{id}\otimes \text{id}\otimes F\otimes \text{id}$ (see \cite[Section 8.2]{cfqt23}).


\begin{lem}\label{intertwiningpropofcom} If $\Delta\colon \cM\to p\tilde{\cM}p$ is as above, then the following hold:

\begin{enumerate}
    \item[(a)] $\Delta(\cM)\not\prec_{\tilde{\cM}}\mathcal{Q}\:\overline{\otimes}\:\cM^{\text{op}}\:\overline{\otimes}\:\cM\:\overline{\otimes}\:\mathbb{M}_n(\C)\:\overline{\otimes}\:\mathbb{M}_n(\C)^{\text{op}},\:\mathcal{M}\:\overline{\otimes}\:\cM^{\text{op}}\:\overline{\otimes}\:\mathcal{Q}\:\overline{\otimes}\:\mathbb{M}_n(\C)\:\overline{\otimes}\:\mathbb{M}_n(\C)^{\text{op}}$ for any von Neumann subalgebra $\mathcal{Q}\subset\cM$ such that $\cM\not\prec_{\cM}\mathcal{Q}$. Similarly, $\Delta(\cM)\not\prec_{\tilde{\cM}}\cM\:\overline{\otimes}\:\mathcal{Q}\:\overline{\otimes}\:\cM\:\overline{\otimes}\:\mathbb{M}_n(\C)\:\overline{\otimes}\:\mathbb{M}_n(\C)^{\text{op}}$ for any von Neumann subalgebra $\mathcal{Q}\subset \cM^{\text{op}}$ with $\cM^{\text{op}}\not\prec_{M^{\text{op}}}\mathcal{Q}$.
    
    \item[(b)] $\Delta(\mathcal{Q})\not\prec\cM\:\overline{\otimes}\:1\:\overline{\otimes}\:\cM\:\overline{\otimes}\:\mathbb{M}_n(\C)\:\overline{\otimes}\:\mathbb{M}_n(\C)$, $\Delta(\mathcal{Q})\not\prec\cM\:\overline{\otimes}\:\cM^{\text{op}}\:\overline{\otimes}\:1\:\overline{\otimes}\:\mathbb{M}_n(\C)\:\overline{\otimes}\:\mathbb{M}_n(\C)$ and $\Delta(\mathcal{Q})\not\prec1\:\overline{\otimes}\:\cM^{\text{op}}\:\overline{\otimes}\:\cM\:\overline{\otimes}\:\mathbb{M}_n(\C)\:\overline{\otimes}\:\mathbb{M}_n(\C)$, for any diffuse von Neumann subalgebra $\mathcal{Q}\subset\cM$.
    
    \item[(c)] If $\mathcal{H}\subset L^2(p\tilde{\cM}p)$ is a $\Delta(\cM)$-sub-bimodule which is right finitely generated, then we have $\mathcal{H}\subset L^2(\Delta(\cM))$. 
\end{enumerate}
\end{lem}

The results in (a) and (b) of the prior statement are obtained by adapting \cite[Lemma 8.5]{cfqt23} to the case of an isomorphism between two twisted group von Neumann algebras. For part (c), the result can be readily adapted to our setting from \cite[Lemma 4.10(c)]{cios22} (see also \cite[Proposition 7.2.3]{ipv10}).

We also recall the notion of height first defined in \cite[Section 4]{ioa10} and \cite[Section 3]{ipv10} and adapted to twisted group von Neumann algebras in  \cite[Section 4]{DV24}. Given a countable group $H$, a 2-cocycle $c\colon H\times H\to\mathbb{T}$ and a subgroup $\mathcal{G}\leqslant \cL_c(H)$, define 

\begin{equation*}
    h_H(\mathcal{G})=\inf_{w\in \mathcal{G}}\left(\max_{h\in H}|\tau(v_{h^{-1}}w)|\right).
\end{equation*}

\noindent In the following statement, we refer to \cite[Lemma 8.6]{cfqt23} and note that its proof can be adapted
to the case of an isomorphism between two twisted group von Neumann algebras.

\begin{lem}\label{positiveheight} Let $G$ and $H$ be countable groups, $d\colon G\times G\to\mathbb{T}$ and $c\colon H\times H\to\mathbb{T}$ be 2-cocycles. Assume $\cM=\cL_d(G)=\cL_c(H)$, let $\Delta$ be as above and assume $n=1$ in the above definition of $\widetilde{\mathcal M}$. If there exist two nonzero elements $x,y \in \widetilde{\mathcal M}$ such that $\Delta(u_g) x =y(u_g\:\otimes\overline{u}_{g^{-1}}\otimes u_g)$ for all $g\in G$, then $h_H(\{u_g:g\in G\})>0$.
\end{lem}

For further use, we also recall the following natural action associated with wreath-like products as defined in \cite[Section 8.1]{cfqt23}. Assume $A$ is a free abelian group, $B\curvearrowright I$ is a faithful action with infinite orbits and let $G\in\mathcal{WR}(A,B\curvearrowright I)$. Denote by $\varepsilon\colon G\to B$ the canonical quotient map for the wreath-like product $G$, and note that we have an action $G\curvearrowright I$ given by $g\cdot i=\varepsilon(g)\cdot i$. Let $d\colon G\times G\to\mathbb{T}$ be a 2-cocycle and define an action $G\curvearrowright^{\gamma}\cL(A^{(I)})$ by 

\begin{equation}\label{actionwithcocycle}
    \gamma_g(\otimes_{i\in I}a_i)=d(g,(a_i)_i)d(g(a_i)_i,g^{-1}) \otimes_{i\in I}a_{g^{-1}\cdot i},
\end{equation}

\noindent for $g\in G$ and $(a_i)_i\in A^{(I)}$. By assumption, we take $d$ to be trivial on $A^{(I)}$, and hence this action naturally extends to an action of the quotient group $B\curvearrowright\cL(A^{(I)})$. 

We note in passing that the action $\sigma$ is \textit{built over} $G\curvearrowright I$, as defined in \cite[Definition 2.5]{kv15}; specifically, it satisfies $\sigma_g(\cL(A^i))=\cL(A^{g\cdot i})$. This property will play a role in the proof of \ref{strongsuperrig}.

\vspace{2mm}

\begin{lem}[{\cite[Lemma 8.1]{cfqt23}}]\label{weakmixandfree} The actions $\gamma$ and $\gamma_{|_B}$ are weak mixing and free.
\end{lem}

We are now ready to prove the main theorem of this section. Although this result builds on previous work \cite{cios22, cfqt23}, we provide all the details for the sake of completeness.

\begin{proof}[\textbf{Proof of \ref{strongsuperrig}}] Identify $\cL_d(G)^t$ and $\cL_{c}(H)$ under the $\ast$-isomorphism $\Psi$. Let $\mathbb{D}_n(\C)\subset\mathbb{M}_n(\C)$ be the subalgebra of diagonal matrices. For $1\leq i\leq n^2$, let $e_i=\mathbbm{1}_{\{i\}}\in\mathbb{D}_n(\C)\:\overline{\otimes}\:\mathbb{D}_n(\C)^{\text{op}}$. Denote

\begin{align*}
    \cM:&=\cL_d(G),\quad \tilde{\cM}:=\cM\:\overline{\otimes}\:\cM^{\text{op}}\:\overline{\otimes}\:\cM\:\overline{\otimes}\:\mathbb{M}_n(\C)\:\overline{\otimes}\:\mathbb{M}_n(\C)^{\text{op}},\quad \mathcal{P}:=\Delta(\cL(A^{(I)})),\\
    &\tilde{\mathcal{Q}}:=\cL(A^{(I)})\:\overline{\otimes}\:\cL(A^{(I)})^{\text{op}}\:\overline{\otimes}\:\cL(A^{(I)})\:\overline{\otimes}\:\mathbb{D}_n\:\overline{\otimes}\:\mathbb{D}_n^{\text{op}},\quad\mathcal{R}:=\mathcal{P}'\cap \tilde{\cM}\subset \tilde{\cM}.
\end{align*}

\vspace{2mm}

\begin{claim}\label{intertwiningcios} $\mathcal{R}\prec_{\tilde{\cM}}^s\tilde{\mathcal{Q}}$.
\end{claim}

\begin{subproof}[Proof of Claim \ref{intertwiningcios}] Since $\cM=\mathbb C\rtimes_{\eta}G$, $B$ is a subgroup of a hyperbolic group, and $\cM$ has property (T), we can use \cite[Theorem 3.10]{cios22} to obtain that $\mathcal{P}\prec^s\cM\:\overline{\otimes}\:\cM^{\text{op}}\:\overline{\otimes}\:\cL(A^{(I)})\:\overline{\otimes}\:\mathbb{M}_n(\C)\:\overline{\otimes}\:\mathbb{M}_n(\C)^{\text{op}}$. Doing the same argument in each tensor, we obtain $\mathcal{P}\prec^s\cL(A^{(I)})\:\overline{\otimes}\:\cM^{\text{op}}\:\overline{\otimes}\:\cM\:\overline{\otimes}\:\mathbb{M}_n(\C)\:\overline{\otimes}\:\mathbb{M}_n(\C)^{\text{op}}$ and $\mathcal{P}\prec^s\cM\:\overline{\otimes}\:\cL(A^{(I)})^{\text{op}}\:\overline{\otimes}\:\cM\:\overline{\otimes}\:\mathbb{M}_n(\C)\:\overline{\otimes}\:\mathbb{M}_n(\C)^{\text{op}}$. By \cite[Lemma 2.8(2)]{dhi19} we have $\mathcal{P}\prec^s \cL(A^{(I)})\:\overline{\otimes}\:\cL(A^{(I)})^{\text{op}}\:\overline{\otimes}\:\cL(A^{(I)})\:\overline{\otimes}\:\mathbb{M}_n(\C)\:\overline{\otimes}\:\mathbb{M}_n(\C)^{\text{op}}$. 

By Lemma \ref{intertwiningpropofcom}(b) we see that $\mathcal{P}$ cannot intertwine in either $\cL(A^{(I)})\:\overline{\otimes}\:\cL(A^{(I)})^{\text{op}}\:\overline{\otimes}\:1\:\overline{\otimes}\:\mathbb{M}_n(\C)\:\overline{\otimes}\:\mathbb{M}_n(\C)^{\text{op}}$, $\:\cL(A^{(I)})\:\overline{\otimes}\:1\:\overline{\otimes}\:\cL(A^{(I)})\:\overline{\otimes}\:\mathbb{M}_n(\C)\:\overline{\otimes}\:\mathbb{M}_n(\C)^{\text{op}}$ or $1\:\overline{\otimes}\:\cL(A^{(I)})^{\text{op}}\:\overline{\otimes}\cL(A^{(I)})\:\overline{\otimes}\:\mathbb{M}_n(\C)\:\overline{\otimes}\:\mathbb{M}_n(\C)^{\text{op}}$. By \cite[Corollary 4.7]{cios22}, $\mathcal{R}$ is amenable. Moreover, since $\mathcal{R}$ is normalized by $\Delta(\cM)$, repeating the first part of the proof we get $\mathcal{R}\prec_{\tilde{\cM}}^s\tilde{\mathcal{Q}}$.
\end{subproof}

Since $\tilde{\cQ}\subset\tilde{\cM}$ is a Cartan subalgebra, combining Claim \ref{intertwiningcios} with \cite[Lemma 3.7]{cios22} (see also \cite{ioa10}), after replacing $\Delta$ with $\text{Ad}(u)\circ\Delta$, for some unitary $u\in\tilde{\cM}$, we may assume  $p\in\tilde{\mathcal{Q}}$ and

\begin{equation}\label{PinQinR}
    \mathcal{P}\subset \tilde{\mathcal{Q}}p\subset\mathcal{R}.
\end{equation}

\noindent Throughout the proof, we will use of the following notation: for every $b\in B$, let $\hat{b}\in G$ so that $\varepsilon(\hat{b})=b$, where $\varepsilon\colon G\to B$ is the canonical quotient map of the wreath-like product $G\in\mathcal{WR}(A,B\curvearrowright I)$.

Fix $b\in B$ and denote by $\gamma_b=\text{Ad}(\Delta(u_{\hat{b}}))$ the action built from $G\curvearrowright A^{(I)}$ and the cocycle $d\colon G\times G\to\mathbb{T}$ as in \eqref{actionwithcocycle}. Then, $\gamma=(\gamma_b)_{b\in B}$ defines an action on $\mathcal{R}$ that leaves $\mathcal{P}$ invariant. Moreover, the restriction $B\curvearrowright^{\gamma}\mathcal{P}$ is free and weakly mixing by Lemma \ref{weakmixandfree}. 

\vspace{2mm}

\begin{claim}\label{weakmixonR} The action $B\curvearrowright^{\gamma}\mathcal{R}$ is weak mixing.
\end{claim}

\begin{subproof}[Proof of Claim \ref{weakmixonR}] Let $\mathcal{H}\subset L^2(\mathcal{R})$ be a finite dimensional $\gamma(B)$-invariant subspace. Let $\mathcal K\subset L^2(p\tilde{\cM}p)$ be the $\|\cdot\|_2$-closure of the linear span of $\mathcal{H}\Delta(\cM)$. Since $\mathcal{H}$ and $\mathcal{P}$ commute, we get that $\mathcal{P}\mathcal{K}=\mathcal{K}$. If $b\in B$, then $\Delta(u_{\hat{b}})\mathcal{H}=\mathcal{H}\Delta(u_{\hat{b}})$ and thus $\Delta(u_{\hat{b}})\mathcal{K}=\mathcal{K}$. Since $G=\{a\hat{b}:a\in A^{(I)},\: b\in B\}$ then $\mathcal{K}$ is a left $\Delta(\cM)$-module. Thus, $\mathcal{K}$ is a $\Delta(\cM)$-bimodule which is right finitely generated since $\mathcal{H}$ is finite dimensional. Using Lemma \ref{intertwiningpropofcom}(c), we obtain $\mathcal{K}\subset L^2(\Delta(\cM))$, and hence $\mathcal{H}\subset L^2(\Delta(\cM))$. Since $\mathcal{H}$ commutes with $\mathcal{P}$, we have $\mathcal{H}\subset L^2(\mathcal{P})$. By Lemma \ref{weakmixandfree} the restriction of $\gamma$ to $\mathcal{P}$ is weak mixing implying that $\mathcal{H}\subset\C p$.
\end{subproof}

From the prior two claims, it follows that $\mathcal{R}$ is an algebra of type  I$_k$, for some $k\in\N$. Using the inclusion \eqref{PinQinR} and the first paragraph of the proof of \cite[Lemma 3.8]{cios22}, we have a decomposition $\mathcal{R}=\mathscr{Z}(\mathcal{R})\:\overline{\otimes}\:\mathbb{M}_k(\C)^{\text{op}}$ such that $\tilde{\mathcal{Q}}p=\mathscr{Z}(\mathcal{R})\:\overline{\otimes}\:\mathbb{D}_k^{\text{op}}$. Thus, $(\mathcal{R})_1\subseteq\sum_{i=1}^k(\tilde{\mathcal{Q}}p)_1x_i$ for some $x_1,...,x_n\in\mathcal{R}$. Moreover, by \cite[Lemma 3.8]{cios22}, there exists an action $\beta=(\beta_b)_{b\in B}$ of $B$ on $\mathcal{R}$ such that

\begin{enumerate}
    \item[(a)] for every $b\in B$ we have $\beta_b=\gamma_b\circ\text{Ad}(w_b)=\text{Ad}(\Delta(u_{\hat{b}})w_b)$, for some $w_b\in \mathscr{U}(\mathcal{R})$;
    \item[(b)] $\tilde{\mathcal{Q}}p$ is $\beta(B)$-invariant and the restriction of $\beta$ to $\tilde{\mathcal{Q}}p$ is free; and
    \item[(c)] the minimal projections $p_1,\ldots,p_{k^2}$ of $1\:\overline{\otimes}\:\mathbb{D}_k\:\overline{\otimes}\:\mathbb{D}_k^{\text{op}}\subset\tilde{\mathcal{Q}}p$ are $\beta(B)$-invariant and the restriction of $\beta$ to $\tilde{\mathcal{Q}}p_i$ is weak mixing for every $1\leq i\leq k^2$.
\end{enumerate}

\noindent Let $\tilde{B}=B\times B\times B$ and consider the action $\alpha:=(\alpha_b)_{b\in\tilde{B}}$ of $\tilde{B}$ on $\cL(A^{(I)})\:\overline{\otimes}\:\cL(A^{(I)})^{\text{op}}\:\overline{\otimes}\:\cL(A^{(I)})$ given by 

\begin{equation*}
    \alpha_{(b_1,b_2,b_3)}=\text{Ad}(u_{\widehat{b_1}}\otimes \overline{u}_{\widehat{b_2^{-1}}}\otimes u_{\widehat{b_3}}),
\end{equation*}

\noindent for $b_1,b_2,b_3\in B$, where the conjugation action is built from $G\curvearrowright A^{(I)}$ and the cocycle $d$ as in \eqref{actionwithcocycle}. Let $(Y,\nu)$ be the dual of $A$ with its respective Haar measure, and let $(X,\mu):=(Y^{I}\times Y^I\times Y^I,\nu^I\times\nu^I\times \nu^I)$. Define $(\tilde{X},\tilde{\mu})=(X\times\Z/n\Z\times \Z/n\Z,\mu\times \lambda\times \lambda)$, where $\lambda$ is the counting measure on $\Z/n\Z$, and identify $\tilde{Q}=L^{\infty}(\tilde{X},\tilde{\mu})$. Let $\tilde{B}\times\Z/n\Z\times\Z/n\Z\curvearrowright^{\tilde{\alpha}}\tilde{\mathcal{Q}}$ be given by 

\begin{equation*}
    \tilde{\alpha}(h,a_1,a_2)(x,a_3,a_4)=(h\cdot x, a_1+a_3,a_2+a_4),
\end{equation*}

\noindent for $h\in \tilde{B}$, $a_1,a_2,a_3,a_4\in\Z/n\Z$. Denote by $\tilde{\alpha}$ still the corresponding p.m.p. action $\tilde{B}\times\Z/n\Z\times\Z/n\Z\curvearrowright(\tilde{X},\tilde{\mu})$. Let $X_0\subseteq \tilde{X}$ be a measurable set for which $p=\mathbbm{1}_{X_0}$. Since $\tilde{\mathcal{Q}}p=L^{\infty}(X_0)$ is $\beta(B)$-invariant, we get a measure preserving action $B\curvearrowright^{\beta}(X_0,\tilde{\mu}|_{X_0})$.

Notice $\tilde{\mathcal{Q}}\subset \tilde{\cM}$ is a Cartan subalgebra and the equivalence relation associated to the inclusion is equal to $\mathscr{R}(\tilde{B}\times\Z/n\Z\times\Z/n\Z\curvearrowright \tilde{X})$. Since the restriction of $\beta$ to $\tilde{\mathcal{Q}}p$ is implemented by unitaries in $\tilde{\cM}$ we deduce that 

\begin{equation}\label{betaxsubsetalphax}
    \beta(B)\cdot x\subset \tilde{\alpha}(\tilde{B}\times\Z/n\Z\times\Z/n\Z)\cdot x\text{, for almost every }x\in X_0.
\end{equation}


\noindent Finally, to apply \cite[Theorem 4.1]{cios22} we prove the following claim.

\vspace{2mm}

\begin{claim}\label{setstozero} Let $B_1=B_2=B_3=B$. For every subgroup of infinite index $K_l\leqslant B_l$ there exists a sequence $(h_m)_m\subset B$ such that for every $s,t\in B$ and $1\leq l\leq 3$ we have

\begin{equation*}
    \tilde{\mu}(\{x\in X:\beta_{h_m}(x)\in\tilde{\alpha}(\tilde{B}_{\ddot{l}}\times sK_lt\times\Z/n\Z\times\Z/n\Z)(x)\})\to 0\text{ as } m\to\infty.
\end{equation*}

\noindent The notation $\tilde{B}_{\ddot{l}}$ means we remove the $l$ component from $\tilde{B}=B_1\times B_2\times B_3$.
\end{claim}

\begin{subproof}[Proof of Claim \ref{setstozero}] Let $G_l:=\varepsilon^{-1}(K_l)$, and notice that $G_l$ is an infinite index subgroup of $G$. By Lemma \ref{intertwiningpropofcom}(a), we can find a sequence $(g_m)_m\subset G$ such that

\begin{align*}
    &\|\mathbb{E}_{\cM\overline{\otimes}\cM^{\text{op}}\overline{\otimes}\cL_d(G_3)\overline{\otimes}\mathbb{M}_n(\C)\overline{\otimes}\mathbb{M}_n(\C)^{\text{op}}}(x\Delta(u_{g_m})y)\|_2,\:\|\mathbb{E}_{\cM\overline{\otimes}\cL_d(G_2)^{\text{op}}\overline{\otimes}\cM\overline{\otimes}\mathbb{M}_n(\C)\overline{\otimes}\mathbb{M}_n(\C)^{\text{op}}}(x\Delta(u_{g_m})y)\|_2,\\
    &\quad\quad\quad\quad\quad\quad\quad\quad\quad\quad\mathbb\|\mathbb{E}_{\cL_d(G_1)\overline{\otimes}\cM^{\text{op}}\overline{\otimes}\cM\overline{\otimes}\mathbb{M}_n(\C)\overline{\otimes}\mathbb{M}_n(\C)^{\text{op}}}(x\Delta(u_{g_m})y)\|_2
\end{align*}

\noindent converge to zero for all $x,y\in \tilde{\cM}$, as $m\ra \infty$. Let $h_m:=\varepsilon(g_m)\in B$. We claim $(h_m)_m$ is the sequence that satisfies the statement. Since $g_m^{-1}\widehat{h_m}\in A^{(I)}$ and $w_{h_m}\in\mathscr{U}(\mathcal{R})$ we get  $\Delta(u_{\widehat{h_m}})w_{h_m}\in\Delta(u_{g_m})\mathscr{U}(\mathcal{R})$. Thus, for every $m$, $\Delta(u_{\widehat{h_m}})w_{h_m}\in \sum_{i=1}^t\Delta(u_{g_m})(\tilde{\mathcal{Q}})_1x_i$ for $x_1,...,x_t\in\mathcal{R}$. In particular, this implies that the following quantities

\begin{align}\label{Deltaw_htozero}
&\|\mathbb{E}_{\cM\overline{\otimes}\cM^{\text{op}}\overline{\otimes}\cL_d(G_3)\overline{\otimes}\mathbb{M}_n(\C)\overline{\otimes}\mathbb{M}_n(\C)^{\text{op}}}(x\Delta(u_{\widehat{h_m}})w_{h_m}y)\|_2,\nonumber\\ &\|\mathbb{E}_{\cM\overline{\otimes}\cL_d(G_2)^{\text{op}}\overline{\otimes}\cM\overline{\otimes}\mathbb{M}_n(\C)\overline{\otimes}\mathbb{M}_n(\C)^{\text{op}}}(x\Delta(u_{\widehat{h_m}})w_{h_m}y)\|_2,\\
&\|\mathbb{E}_{\cL_d(G_1)\overline{\otimes}\cM^{\text{op}}\overline{\otimes}\cM\overline{\otimes}\mathbb{M}_n(\C)\overline{\otimes}\mathbb{M}_n(\C)^{\text{op}}}(x\Delta(u_{\widehat{h_m}})w_{h_m}y)\|_2\nonumber
\end{align}

\noindent converge to zero for every $x,y\in\tilde{\mathcal{M}}$ as $m\ra \infty$. On the other hand, $\beta_{h_m}=\text{Ad}(\Delta(u_{\widehat{h_m}})w_{h_m})$ and $\alpha(b_1,b_2,b_3)=\text{Ad}(u_{(\widehat{b_1},\widehat{b_2^{-1}},\widehat{b_3})})$ for all $b_1,b_2,b_3\in B$. Altogether, these imply each of the following 

\begin{align*}
    \|\mathbb{E}_{\cM\overline{\otimes}\cM^{\text{op}}\overline{\otimes}\cL_d(G_3)\overline{\otimes}\mathbb{M}_n(\C)\overline{\otimes}\mathbb{M}_n(\C)^{\text{op}}}&((1\otimes 1\otimes u_{\hat{s}}^*)\Delta(u_{\widehat{h_m}})w_{h_m}(1\otimes 1 \otimes u_{\hat{t}}^*))\|_2^2\\
    &=\mu(\{x\in X:\beta_{h_m}(x)\in\tilde{\alpha}(\tilde{B}_{\ddot{3}}\times sK_3t\times\Z/n\Z\times\Z/n\Z)(x)\}),
\end{align*}

\begin{align*}
    \|\mathbb{E}_{\cM\overline{\otimes}\cL_d(G_2)^{\text{op}}\overline{\otimes}\cM\overline{\otimes}\mathbb{M}_n(\C)\overline{\otimes}\mathbb{M}_n(\C)^{\text{op}}}&((1\otimes \overline{u_{\hat{t}}}^*\otimes 1)\Delta(u_{\widehat{h_m}})w_{h_m}(1\otimes \overline{u_{\hat{s}}}^*\otimes 1))\|_2^2\\
    &=\mu(\{x\in X:\beta_{h_m}(x)\in\tilde{\alpha}(\tilde{B}_{\ddot{2}}\times sK_2t\times\Z/n\Z\times\Z/n\Z)(x)\}),
\end{align*}

\begin{align*}
    \|\mathbb{E}_{\cL_d(G_1)\overline{\otimes}\cM^{\text{op}}\overline{\otimes}\cM\overline{\otimes}\mathbb{M}_n(\C)\overline{\otimes}\mathbb{M}_n(\C)^{\text{op}}}&((u_{\hat{s}}^*\otimes 1\otimes 1)\Delta(u_{\widehat{h_m}})w_{h_m}(u_{\hat{t}}^*\otimes 1\otimes 1))\|_2^2\\
    &=\mu(\{x\in X:\beta_{h_m}(x)\in\tilde{\alpha}(\tilde{B}_{\ddot{1}}\times sK_1t\times\Z/n\Z\times\Z/n\Z)(x)\})
\end{align*}

\noindent converge to zero, as $m\ra \infty$. These yield the desired conclusion.
\end{subproof}

Fix $j\in I$, so that $j\in I_l$ for some $1\leq l\leq 3$, and let $b=(b_1,b_2,b_3)\in \tilde{B}\setminus\{1\}$. Then, $\text{Stab}_{\tilde{B}}(j)=\tilde{B}_{\ddot{l}}\times \text{Stab}_{B_l}(j)$ and $C_{\tilde{B}}(b)=C_{B_1}(b_1)\times C_{B_2}(b_2)\times C_{B_3}(b_3)$. Since $b\neq 1$, there exists $1\leq l\leq 3$ for which $b_l\neq 1$, and hence, $C_{\tilde{B}}(b)\leqslant \tilde{B}_{\ddot{l}}\times C_{B_l}(b_l)$, where $C_{B_l}(b_l)$ is an infinite index subgroup of $B_l$. 

Fix $1\leq i\leq k^2$. Let $X_i\subset X_0$ be a $\beta(B)$-invariant measurable set for which $p_i=\mathbbm{1}_{X_i}$. Since $\beta|_{X_i}$ is free, weak mixing, $B$ has property (T), equation \eqref{betaxsubsetalphax}, the prior paragraph and Claim \ref{setstozero}, we can apply \cite[Theorem 4.1]{cios22} to obtain that $\tilde{\mu}(X_i)=1$ and there exists $\theta_i\in[\mathscr{R}(\tilde{B}\times\Z/n\Z\times\Z/n\Z\curvearrowright^{\tilde{\alpha}}\tilde{X})]$ and an injective group homomorphism $\rho_i=(\rho_{i,1},\rho_{i,2},\rho_{i,3}):B\to \tilde{B}$ such that $\theta_i(X_i)=X\times\{i\}\equiv X$ and $\theta_i\circ\beta(b)|_{X_i}=\tilde{\alpha}(\rho_i(b))\circ\theta_i|_{X_i}$, for all $b\in B$.

Now, let $u_i\in\mathcal N_{\tilde{\mathcal M}}(\tilde{\mathcal{Q})}$ such that $u_iau_i^*=a\circ\phi_i^{-1}$, for every $a\in\tilde{\mathcal{Q}}$. Then, $u_ip_iu_i^*=1\otimes 1\otimes 1\otimes e_i$ and the last relation implies that we can find $(\zeta_{i,b})_{b\in B}\subset\mathscr U(\cL(A^{(I)})\:\overline{\otimes}\:\cL(A^{(I)})^{\text{op}}\:\overline{\otimes}\:\cL(A^{(I)}))$ such that

\begin{equation}\label{conjugat}
    u_i\Delta(u_{\widehat{b}})w_bp_iu_i^*=\zeta_{i,b}u_{(\widehat{\rho_{i,1}(b)},\widehat{\rho_{i,2}(b)^{-1}},\widehat{\rho_{i,3}(b)})}\otimes e_i,\quad \text{for every $b\in B$.}
\end{equation}

\vspace{2mm}

\begin{claim}\label{epsilonsconjugated} $\rho_i$ is conjugate to $\rho_j$ for all $1\leq i,j\leq k^2$.
\end{claim}

\begin{subproof}[Proof of Claim \ref{epsilonsconjugated}] We first show $B_0=\rho_{i,1}(B)$ has finite index in $B$. If this is not true, then $G_0=\varepsilon^{-1}(B_0)$ has infinite index in $G$. On the other hand, equation \eqref{conjugat} gives that $\Delta(\cM)\prec_{\tilde{\cM}}\cL_d(G_0)\:\overline{\otimes}\:\cM^{\text{op}}\:\overline{\otimes}\:\cM\:\overline{\otimes}\:\mathbb{M}_n(\C)\:\overline{\otimes}\:\mathbb{M}_n(\C)^{\text{op}}$, which contradicts Lemma \ref{intertwiningpropofcom}(a). Similarly, $\rho_{i,2}(B)$ and $\rho_{i,3}(B)$ are finite index in $B$.

Let $1\leq i,j\leq k^2$. Since $p_i,p_j\in\mathcal{R}$ are equivalent projections, as they are minimal projections, then $p_j=zp_iz^*$ for $z\in\mathscr{U}(\mathcal{R})$. As $\Delta(u_{\hat{b}})w_b\in\mathscr{U}(p\tilde{\mathcal{M}}p)$ normalizes $\mathcal{R}$, we get that $z_b:=\text{Ad}(\Delta(u_{\hat{b}})w_b)(z)\in\mathscr{U}(\mathcal{R})$. Then,

\begin{equation*}
    \Delta(u_{\hat{b}})w_bp_j=\Delta(u_{\hat{b}})w_bzp_iz^*=z_b\Delta(u_{\hat{b}})w_bp_iz^*.
\end{equation*}

\noindent Using \eqref{conjugat} we get that

\begin{align}\label{conjugatedepsilonsij}
    \zeta_{j,b}(u_{(\widehat{\rho_{j,1}(b)},\widehat{\rho_{j,2}(b)^{-1}},\widehat{\rho_{j,3}(b)})}\otimes e_j)&=u_j\Delta(u_{\hat{b}})w_bp_ju_j^*=u_j(z_b\Delta(u_{\hat{b}})w_bp_iz^*)u_j^*\nonumber\\
    &=u_j(z_b(u_i^*(\zeta_{i,b}u_{(\widehat{\rho_{i,1}(b)},\widehat{\rho_{i,2}(b)^{-1}},\widehat{\rho_{i,3}(b)})}\otimes e_i)u_i)z^*)u_j^*,
\end{align}

\noindent for all $b\in B$. For $b\in B$, denote by $\tilde{\xi}_b=u_j(z_bu_i^*(\zeta_{i,b}u_{(\widehat{\rho_{i,1}(b)},\widehat{\rho_{i,2}(b)^{-1}},\widehat{\rho_{i,3}(b)})}\otimes e_i)u_iz^*)u_j^*$. For a finite subset $F\subset B\times B\times B$, denote by $P_F$ the orthogonal projection from $L^2(\tilde{\cM})$ onto the $\|\cdot\|_2$-closure of the linear span of 

\begin{equation*}
    \{xu_g\otimes y:x\in \cL(A^{(I)})\:\overline{\otimes}\:\cL(A^{(I)})^{\text{op}}\:\overline{\otimes}\:\cL(A^{(I)}),\: g\in (\varepsilon\times\varepsilon\times\varepsilon)^{-1}(F),\: y\in\mathbb{M}_n(\C)\:\overline{\otimes}\:\mathbb{M}_n(\C)^{\text{op}}\}.
\end{equation*}

\noindent Using that $\cL(A^{(I)})\:\overline{\otimes}\:\cL(A^{(I)})^{\text{op}}\:\overline{\otimes}\:\cL(A^{(I)})\subset\cM\:\overline{\otimes}\:\cM^{\text{op}}\:\overline{\otimes}\:\cM$ is a Cartan subalgebra and approximating $u_j,u_i,z$ in $\|\cdot\|_2$, we can find a finite set $F\subset B\times B\times B$ such that $P_{F\rho_i(b)F}(\tilde{\xi}_h)\neq 0$, for all $b\in B$. Since $\zeta_{j,b}\in\mathscr{U}(\cL(A^{(I)})\:\overline{\otimes}\:\cL(A^{(I)})^{\text{op}}\:\overline{\otimes}\:\cL(A^{(I)}))$, \eqref{conjugatedepsilonsij} implies that $\rho_j(b)\in F\rho_i(b)F$ for all $b\in B$. If $g\in B\setminus\{1\}$, then as $B$ is ICC and $\rho_{i,1}(B),\rho_{i,2}(B),\rho_{i,3}(B)$ have finite index in $B$, the set $\{\rho_{i,k}(b)g\rho_{i,k}(b)^{-1}:b\in B\}$ is infinite for every $k=1,2,3$. Thus, the set $\{\rho_i(b)g\rho_i(b)^{-1}:b\in B\}$ is infinite for $g\in\tilde{B}\setminus\{(1,1,1)\}$. By \cite[Lemma 7.1]{bv13}, there exists $g\in B\times B\times B$ with $\rho_j(b)=g\rho_i(b)g^{-1}$ for all $b\in B$.
\end{subproof}

The prior claim gives a homomorphism $\delta=(\delta_1,\delta_2,\delta_3):B\to B\times B\times B$ such that for every $1\leq i\leq k^2$, there exists $g_i\in B\times B\times B$ satisfying $\rho_i(b)=g_i\delta(b)g_i^{-1}$ for all $b\in B$. Hence, after replacing $\theta_i$ with $\alpha(g_i^{-1})\circ\theta_i$, we may assume that $\rho_i=\delta$ and 

\begin{equation*}
    u_i\Delta(u_{\hat{b}})w_bp_iu_i^*=\zeta_{i,h}(u_{(\widehat{\delta_1(b)},\widehat{\delta_2(b)^{-1}},\widehat{\delta_3(b)})}\otimes e_i),\text{ for all }1\leq i\leq k^2,\: b\in B.
\end{equation*}

\noindent Let $u=\sum_{i=1}^{k^2}u_ip_i$, $e=\sum_{i=1}^{k^2}e_i$. Then, $u$ is a partial isometry with $uu^*=1\otimes 1\otimes 1\otimes e$, $u^*u=p$, and $u\tilde{\mathcal{Q}}pu^*=\tilde{\mathcal{Q}}(1\otimes 1\otimes 1\otimes e)$. If $\zeta_b=\sum_{i=1}^{k^2}\zeta_{i,b}\otimes e_i\in\mathscr{U}(\tilde{\mathcal{Q}}(1\otimes 1\otimes 1\otimes e))$, then 

\begin{equation*}
    u\Delta(u_{\hat{b}})w_bu^*=\zeta_b(u_{(\widehat{\delta_1(b)},\widehat{\delta_2(b)^{-1}},\widehat{\delta_3(b)})}\otimes e),\text{ for all }b\in B.
\end{equation*}

\noindent In particular, $t^2=(\tau\otimes\tau^{\text{op}}\otimes\tau\otimes \text{Tr}\otimes\text{Tr}^{\text{op}})(p)=(\text{Tr}\otimes\text{Tr}^{\text{op}})(e)=k^2$. So, $n=t=k$, $e=1$ and $p=1\otimes 1\otimes 1\otimes 1\otimes 1$. Hence, replacing $\Delta$ with $\text{Ad}(u)\circ\Delta$ we may assume that $(\zeta_b)_{b\in B}\subset\mathscr{U}(\tilde{\mathcal{Q}})$ and

\begin{equation}\label{conjug}
    \Delta(u_{\widehat{b}})w_b=\zeta_{b}(u_{(\widehat{\delta_{1}(b)}}\otimes \overline{u}_{\widehat{\delta_{2}(b)^{-1}}}\otimes u_{\widehat{\delta_{3}(b)})}\otimes 1\otimes 1),\quad \text{for every $b\in B$.}
\end{equation}

\vspace{2mm}

\begin{claim}\label{RcontainedtildeQ}
    $\mathcal R \subseteq \cL(A^{(I)})\:\overline{\otimes}\:\cL(A^{(I)})^{\text{op}}\:\overline{\otimes}\:\cL(A^{(I)})\:\overline{\otimes}\:\mathbb{M}_n(\C)\:\overline{\otimes}\:\mathbb{M}_n(\C)^{\text{op}}$.
\end{claim} 

\begin{subproof}[Proof of Claim \ref{RcontainedtildeQ}] Since $(\Delta(u_{\hat{b}})w_b)_{b\in B}\subset\mathscr{U}(\tilde{\mathcal{M}})$ normalizes $\mathcal{R}$ and $(\zeta_b)_{b\in B}\subset\mathscr{U}(\tilde{\mathcal{Q}})$, equation \eqref{conjug} shows that $(u_{(\widehat{\delta_{1}(b)},\widehat{\delta_{2}(b)^{-1}},\widehat{\delta_{3}(b)})}\otimes 1\otimes 1)_{b\in B}$ normalizes $\mathcal{R}$. Thus, to prove the claim it suffices to argue that for every $x,y\in\cM\:\overline{\otimes}\:\cM^{\text{op}}\:\overline{\otimes}\:\cM$ and $z\in \left (\cM\:\overline{\otimes}\:\cM^{\text{op}}\:\overline{\otimes}\:\cM\right )\ominus \left (\cL(A^{(I)})\:\overline{\otimes}\:\cL(A^{(I)})^{\text{op}}\:\overline{\otimes}\:\cL(A^{(I)})\right)$ the sequence $(h_m)_m\subset B$ from Claim \ref{setstozero} satisfies

\begin{equation*}
    \|\mathbb{E}_{\cL(A^{(I)})\:\overline{\otimes}\:\cL(A^{(I)})^{\text{op}}\:\overline{\otimes}\:\cL(A^{(I)})}(xu_{(\widehat{\delta_{1}(h_m)},\widehat{\delta_{2}(h_m)^{-1}},\widehat{\delta_{3}(h_m)})}zu_{(\widehat{\delta_{1}(h_m)},\widehat{\delta_{2}(h_m)^{-1}},\widehat{\delta_{3}(h_m)})}^*y)\|_2\to 0.
\end{equation*}

\noindent Assume $x=u_a$, $y=u_b$ and $z=u_g$ for $a,b,g\in G\times G\times G$ and $g\not\in A^{(I)}\times A^{(I)}\times A^{(I)}$. Write $(\varepsilon\times\varepsilon\times\varepsilon)(a)=(a_1,a_2,a_3)$, $(\varepsilon\times\varepsilon\times\varepsilon)(b)=(b_1,b_2,b_3)$ and $(\varepsilon\times\varepsilon\times\varepsilon)(g)=(g_1,g_2,g_3)$. Since $(g_1,g_2,g_3)\neq (1,1,1)$, we have that $g_i\neq 1$ for some $i=1,2,3$.

Suppose $g_1\neq 1$. For $h\in B$ denote 

$$s_h=\|\mathbb{E}_{\cL(A^{(I)})\:\overline{\otimes}\:\cL(A^{(I)})^{\text{op}}\:\overline{\otimes}\:\cL(A^{(I)})}(xu_{(\widehat{\delta_{1}(h)},\widehat{\delta_{2}(h)^{-1}},\widehat{\delta_{3}(h)})}zu_{(\widehat{\delta_{1}(h)},\widehat{\delta_{2}(h)^{-1}},\widehat{\delta_{3}(h)})}^*y)\|_2.$$

\noindent If $s_h=0$ for all $h\in B$, then the assertion follows. Otherwise, if $s_h\neq 0$ for $h\in B$, then $a_1\delta_1(h)g_1\delta_1(h)^{-1}b_1=1$. In particular, there exists $l\in B$ with $a_1lg_1l^{-1}b_1=1$. If $s_{h_m}\neq 0$ for some $m\in\N$, then $\delta_1(h_m)g_1\delta_1(h_m)^{-1}=a_1^{-1}b_1^{-1}=lg_1l^{-1}$ and therefore, $\delta_1(h_m)\in lB_0$ for $B_0=C_B(g_1)$. 

Let $G_0=\varepsilon^{-1}(B_0)$. By \eqref{conjug} we obtain that $\Delta(u_{h_m})w_{h_m}\in (u_{\hat{l}}\otimes 1\otimes 1)(\cL_d(G_0)\:\overline{\otimes}\:\cL_d(G)^{\text{op}}\:\overline{\otimes}\:\cL_d(G))$ for any such $m$. Since $g_1\neq 1$ and $B$ is ICC, $B_0\leqslant B$ is infinite index. Thus, \eqref{Deltaw_htozero} implies that $\{m\in\N:s_{h_m}\neq 0\}$ is finite, proving that $s_{h_m}\to 0$.

The claim follows similarly if $g_3\neq 1$. Now, assume $g_2\neq 1$. Suppose $s_{h_m}\neq 0$ for some $m\in\N$. Then, $b_2\delta_2(h_m)g_2\delta_2(h_m)^{-1}a_2=1$. In particular, there exists $f\in B$ with $b_2fg_2f^{-1}a_2=1$, and hence, $\delta_2(h_m)g_2\delta_2(h_m)^{-1}=b_2^{-1}a_2^{-1}=fg_2f^{-1}$, and therefore, $\delta_2(h_m)\in fB_0$ for $B_0=C_B(g_2)$. Let $G_0=\varepsilon^{-1}(B_0)$. By \eqref{conjug} we obtain that $\Delta(u_{\widehat{h_m}})w_{h_m}\in (\cL_d(G)\:\overline{\otimes}\:\cL_d(G_0)^{\text{op}}\overline{\otimes}\:\cL_d(G))(1\otimes u_{\widehat{f}}\otimes 1)$ for any such $m$. Thus, \eqref{Deltaw_htozero} implies that $\{m\in\N:s_{h_m}\neq 0\}$ is finite, proving that $s_{h_m}\to 0$, as $m\ra \infty$.
\end{subproof}

Next, the prior claim implies $w_b\in\cL(A^{(I)})\:\overline{\otimes}\:\cL(A^{(I)})^{\text{op}}\:\overline{\otimes}\:\cL(A^{(I)})\:\overline{\otimes}\:\mathbb{M}_n(\C)\:\overline{\otimes}\:\mathbb{M}_n(\C)^{\text{op}}$ for every $b\in B$. Thus, $\eta_b:=\zeta_b\text{Ad}(u_{(\widehat{\delta_{1}(b)},\widehat{\delta_{2}(b)^{-1}},\widehat{\delta_{3}(b)})})(w_b^*)\in\cL(A^{(I)})\:\overline{\otimes}\:\cL(A^{(I)})^{\text{op}}\:\overline{\otimes}\:\cL(A^{(I)})\:\overline{\otimes}\:\mathbb{M}_n(\C)\:\overline{\otimes}\:\mathbb{M}_n(\C)^{\text{op}}$ and also 

\begin{equation*}
    \Delta(u_{\hat{b}})=\eta_b(u_{\widehat{\delta_1(b)}}\otimes \overline u_{\widehat{\delta_2(b)^{-1}}}\otimes u_{\widehat{\delta_3(b)}}\otimes 1\otimes 1),\quad\text{for every } b\in B.
\end{equation*}

\vspace{2mm}

\begin{claim}\label{assumemapsareidentity} We may assume that $\delta_1=\delta_2=\delta_3={\rm Id}_B$.
\end{claim}

\begin{subproof}[Proof of Claim \ref{assumemapsareidentity}] First, we argue that we may assume that $\delta_1=\delta_3=\text{Id}_B$. Using the flip automorphism $x\otimes y\otimes z\mapsto z\otimes y \otimes x$ and the same argument form the first part of Step 6 in the proof of \cite[Theorem 1.3]{cios22} we obtain that $\delta_1$ is conjugated to $\delta_3$ by a group element $g\in B$. Hence, after replacing $\Delta$ with $\text{Ad}(1\otimes 1\otimes u_{g^{-1}}\otimes 1\otimes 1)\circ \Delta$ and $\eta_b$ by $\text{Ad}(1\otimes 1\otimes u_{g^{-1}}\otimes 1\otimes 1)(\eta_b)d(g^{-1},g\widehat{\delta_1(b)}g^{-1})d(\widehat{\delta_1(b)}g^{-1},g)$, we may assume that $\delta_1=\delta_3=:\delta$ and 

\begin{equation*}
    \Delta(u_{\hat{b}})=\eta_b(u_{\widehat{\delta(b)}}\otimes \overline u_{\widehat{\delta_2(b)^{-1}}}\otimes u_{\widehat{\delta(b)}}\otimes 1\otimes 1),\quad\text{for all } b\in B.
\end{equation*}

\noindent Following a similar argument as in \cite[Theorem 1.3, Step 6]{cios22}, we let $X_1=(U\otimes 1\otimes 1\otimes 1\otimes 1)^*(\Delta_0\otimes\text{id}\otimes\text{id})(U^*)$, $X_{1,b}=(\Delta_0\otimes\text{id}\otimes\text{id})(U\psi(\eta_b\otimes 1))(U\psi(\eta_{\delta(b)}\otimes 1)\otimes 1\otimes 1\otimes 1\otimes 1)$, $X_2=(1\otimes 1\otimes 1\otimes 1\otimes U)^*(\text{id}\otimes\text{id}\otimes\Delta_0)(U^*)$ and $X_{2,b}=(\text{id}\otimes\text{id}\otimes\Delta_0)(U\psi(\eta_b\otimes 1))(1\otimes 1\otimes 1\otimes 1\otimes U\psi(\eta_{\delta(b)}\otimes 1))$, and we obtain 

\begin{align*}
    (\Delta_0\otimes\text{id}\otimes\text{id})\Delta_0(u_{\hat{b}}\otimes 1)=X_{1,b}(u_{\widehat{\delta(\delta(b))}}\otimes 1\otimes \overline{u}_{\widehat{\delta_2(\delta(b))^{-1}}}\otimes 1\otimes u_{\widehat{\delta(\delta(b))}}\otimes 1\otimes \overline{u}_{\widehat{\delta_2(b)^{-1}}}\otimes 1 \otimes u_{\widehat{\delta(b)}}\otimes 1)X_1,
\end{align*}

\noindent and

\begin{align*}
    (\text{id}\otimes\text{id}\otimes\Delta_0)\Delta_0(u_{\hat{b}}\otimes 1)=X_{2,b}(u_{\widehat{\delta(b)}}\otimes 1 \otimes \overline{u}_{\widehat{\delta_2(b)^{-1}}}\otimes 1\otimes u_{\widehat{\delta(\delta(b))}}\otimes 1\otimes \overline{u}_{\widehat{\delta_2(\delta(b))^{-1}}}\otimes 1\otimes u_{\widehat{\delta(\delta(b))}}\otimes 1)X_2.
\end{align*}

\noindent Thus \eqref{flipproperty} in combination with the two relations above show for every $b\in B$ we have

\begin{align}\label{equaltensors1}
    (u_{\widehat{\delta(\delta(b))}}&\otimes 1 \otimes \overline u_{\widehat{\delta_2(\delta(b))^{-1}}}\otimes 1\otimes u_{\widehat{\delta(\delta(b))}}\otimes 1\otimes  \overline u_{\widehat{\delta_2(b)^{-1}}}\otimes 1\otimes u_{\widehat{\delta(b)}}\otimes 1)X_1X_2^*\nonumber\\
    &=X_{1,b}^*X_{2,b}(u_{\widehat{\delta(b)}}\otimes 1\otimes \overline u_{\widehat{\delta_2(b)^{-1}}}\otimes 1\otimes u_{\widehat{\delta(\delta(b))}}\otimes 1\otimes \overline u_{\widehat{\delta_2(\delta(b))^{-1}}}\otimes 1\otimes u_{\widehat{\delta(\delta(b))}}\otimes 1).
\end{align}

\noindent Denote by $\mathscr M :=\cM\:\overline{\otimes}\:\mathbb{M}_n(\C)\:\overline{\otimes}\:\cM^{\text{op}}\:\overline{\otimes}\:\mathbb{M}_n(\C)^{\text{op}}\:\overline{\otimes}\:\cM\:\overline{\otimes}\:\mathbb{M}_n(\C)\:\overline{\otimes}\:\cM^{\text{op}}\:\overline{\otimes}\:\mathbb{M}_n(\C)^{\text{op}}\:\overline{\otimes}\:\cM\:\overline{\otimes}\:\mathbb{M}_n(\C)$ and $\mathscr Q:=\cL(A^{(I)})\:\overline{\otimes}\:\mathbb{M}_n(\C)\:\overline{\otimes}\:\cL(A^{(I)})^{\text{op}}\:\overline{\otimes}\:\mathbb{M}_n(\C)^{\text{op}}\:\overline{\otimes}\:\cL(A^{(I)})\:\overline{\otimes}\:\mathbb{M}_n(\C)\:\overline{\otimes}\:\cL(A^{(I)})^{\text{op}}\:\overline{\otimes}\:\mathbb{M}_n(\C)^{\text{op}}\:\overline{\otimes}\:\cL(A^{(I)})\:\overline{\otimes}\:\mathbb{M}_n(\C)$. Also let $\vardbtilde{B}:=B \times B \times B\times B \times B$.

For a set $F=F_1\times \cdots\times F_5\subset \vardbtilde{B}$, let $\mathcal{H}_{F}$ be the $\|\cdot\|_2$-closure of the linear span of 

\begin{equation*}
    \{u_{g_1}\otimes x_1\otimes \overline{u}_{g_2}\otimes \overline{x}_2\otimes u_{g_3}\otimes x_3\otimes \overline{u}_{g_4}\otimes \overline{x}_4\otimes u_{g_5}\otimes x_5:g_i\in\varepsilon^{-1}(F_i),x_i\in\mathbb{M}_n(\C), 1\leq i\leq 5\}
\end{equation*}

\noindent and $P_{F}$ be the orthogonal projection from $L^2(\mathscr M)$ onto $\mathcal{H}_{F}$. For later use, notice that for every pair of subsets $G=G_1\times \cdots \times G_5, \: F =F_1\times \cdots \times F_5\subset \vardbtilde{B}$ the following hold:

\begin{enumerate}
    \item\label{orthogonalproj} If $F_j\cap G_j =\emptyset$ for some $1\leq j\leq 5$ then  we have $P_F P_G =0$.
    \item\label{secondproperty} For every $(g_1,\ldots,g_5),\: (h_1,\ldots, h_5)\in \vardbtilde{B}$ we have  
    
    \begin{align*}
        (u_{g_1}\otimes 1\otimes &\:\overline u_{g_2} \otimes 1\otimes u_{g_3} \otimes 1 \otimes \overline u_{g_4} \otimes 1\otimes u_{g_5}\otimes 1)\mathcal H_F (u_{h_1}\otimes 1\otimes \overline u_{h_2} \otimes 1\otimes u_{h_3} \otimes 1\otimes \overline u_{h_4} \otimes 1\otimes u_{h_5}\otimes 1)\\
        &=\mathcal H_{g_1F_1h_1\times h_2F_2g_2\times g_3F_3h_3\times h_4F_4g_4\times g_5F_5h_5}.
    \end{align*}
\end{enumerate}

\noindent Next, since $\mathcal{H}_{F}$ is a $\mathscr Q$-bimodule, using the definitions of $\eta_b$, $X_{1,b}$, $X_{2,b}$, $X_1$, and $X_2$ one can find a finite set $F=F_1\times \cdots\times F_5\subset \vardbtilde{B}$ such that 

\begin{equation*}
    \| X_1X_2^*-P_F(X_1X_2^*)\|_2<1/4, \text{ and } \|X_{1,b}^*X_{2,b}-P_{F}(X_{1,b}^*X_{2,b})\|_2<1/4,\text{ for every }b\in B.
\end{equation*} 

\noindent These inequalities combined with \eqref{equaltensors1} show that for all $b\in B$ we have 

\begin{align*}
    \langle P_{F}&(X_{1,b}^*X_{2,b})(u_{\widehat{\delta(b)}}\otimes 1\otimes \overline u_{\widehat{\delta_2(b)^{-1}}}\otimes 1\otimes u_{\widehat{\delta(\delta(b))}}\otimes 1\otimes \overline u_{\widehat{\delta_2(\delta(b))^{-1}}}\otimes1\otimes u_{\widehat{\delta(\delta(b))}}\otimes 1),\\
    &\quad\quad (u_{\widehat{\delta(\delta(b))}}\otimes 1\otimes \overline u_{\widehat{\delta_2(\delta(b))^{-1}}}\otimes 1\otimes u_{\widehat{\delta(\delta(b))}}\otimes 1\otimes\overline u_{\widehat{\delta_2(b)^{-1}}}\otimes 1\otimes u_{\widehat{\delta(b)}}\otimes 1)P_{F}(X_1X_2^*)\rangle \geq1/2.
\end{align*}

\noindent The prior relation together with properties \eqref{orthogonalproj} and \eqref{secondproperty} above imply that $\delta(\delta(b))F_1\cap  F_1\delta(b)\neq \emptyset$, for every $b\in B$. Since $\delta(B)\leqslant B$ has finite index and $B$ is ICC, by \cite[Lemma 7.1]{bv13} one can find $l\in B$ such that $\delta(b)=lbl^{-1}$, for every $b\in B$. Thus, after replacing $\Delta$ with $\text{Ad}(u_{l^{-1}}\otimes 1\otimes u_{l^{-1}}\otimes 1\otimes)\circ\Delta$ and $\eta_b$ with $\text{Ad}(u_{l^{-1}}\otimes 1\otimes u_{l^{-1}}\otimes 1\otimes)(\eta_b)d(l^{-1},lbl^{-1})^2d(bl^{-1},l)^2$, we have

\begin{equation*}
    \Delta(u_{\hat{b}})=\eta_b(u_{\hat{b}}\otimes \overline u_{\widehat{\delta_2(b)^{-1}}}\otimes u_{\hat{b}}\otimes 1\otimes 1),\quad\text{for all }b\in B.
\end{equation*}

\noindent To show $\delta_2$ can be conjugated to the identity we follow the argument from \cite[Claim 8.11]{cfqt23}. 

\begin{align*}
    F_{1,2}\:\circ &\:(\Delta_0\otimes\text{id}\otimes\text{id})\Delta_0(u_{\hat{b}}\otimes 1)\\
    &=F_{1,2}(X_{1,b})(u_{\widehat{\delta_2(b)^{-1}}}^*\otimes 1\otimes \overline {u_{\hat{b}}^*}\otimes 1\otimes u_{\hat{b}}\otimes 1\otimes \overline u_{{\widehat{\delta_2(b)^{-1}}}}\otimes 1\otimes u_{\hat{b}}\otimes 1)F_{1,2}(X_{1}).
\end{align*}

\noindent Similarly, 

\begin{align*}
    F_{3,4}\:\circ &\: (\text{id}\otimes\text{id}\otimes\Delta_0)\Delta_0(u_{\hat{b}}\otimes 1)\\
    &=F_{3,4}(X_{2,b})(u_{\hat{b}}\otimes 1\otimes \overline u_{\widehat{\delta_2(b)^{-1}}}\otimes 1\otimes u_{\widehat{\delta_2(b)^{-1}}}^*\otimes 1\otimes \overline{u_{\hat{b}}^*}\otimes 1\otimes u_{\hat{b}}\otimes 1)F_{3,4}(X_{2}).
\end{align*}

\noindent Using these relations together with \eqref{flipproperty} we get that for all $b\in B$ we have 

\begin{align}\label{equaltensor3}
    (u_{\widehat{\delta_2(b)^{-1}}}^*&\otimes 1\otimes \overline {u_{\hat{b}}^*}\otimes 1\otimes u_{\hat{b}}\otimes 1\otimes \overline u_{{\widehat{\delta_2(b)^{-1}}}}\otimes 1\otimes u_{\hat{b}}\otimes 1) F_{1,2}(X_{1})F_{3,4}(X_{2})^*\nonumber\\
    &= F_{1,2}(X_{1,b})^*F_{3,4}(X_{2,b}) (u_{\hat{b}}\otimes 1\otimes \overline u_{\widehat{\delta_2(b)^{-1}}}\otimes 1\otimes u_{\widehat{\delta_2(b)^{-1}}}^*\otimes 1\otimes \overline{u_{\hat{b}}^*}\otimes 1\otimes u_{\hat{b}}\otimes 1).
\end{align}

\noindent Let $F=F_1\times \cdots\times F_5\subset \vardbtilde{B}$ be a finite set for which $\|F_{1,2}(X_1)F_{3,4}(X_2)^*-P_F(F_{1,2}(X_1)F_{3,4}(X_2)^*)\|_2<1/4$ and $\|F_{1,2}(X_{1,b})^*F_{3,4}(X_{2,b})-P_F(F_{1,2}(X_{1,b})^*F_{3,4}(X_{2,b}))\|_2<1/4$ for all $b\in B$. As in the prior case, we obtain

\begin{align*}
    \langle P_{F}&(F_{1,2}(X_{1,b})^*F_{3,4}(X_{2,b}))(u_{\hat{b}}\otimes 1\otimes \overline u_{\widehat{\delta_2(b)^{-1}}}\otimes 1\otimes u_{\widehat{\delta_2(b)^{-1}}}^*\otimes 1\otimes \overline u_{\hat{b}}^*\otimes 1\otimes u_{\hat{b}})\otimes 1,\\
    &(u_{\widehat{\delta_2(b)^{-1}}}^*\otimes 1\otimes \overline u_{\hat{b}}^*\otimes 1\otimes u_{\hat{b}}\otimes 1\otimes\overline u_{\widehat{\delta_2(b)^{-1}}}\otimes 1\otimes u_{\hat{b}}\otimes 1)P_{F}(F_{1,2}(X_{1})F_{3,4}(X_{2})^*)\rangle \geq 1/2.
\end{align*}

\noindent Therefore, using the prior equation and properties \eqref{orthogonalproj} and \eqref{secondproperty}, we obtain that $F_1b\cap \delta_2(b)F_1\neq\emptyset$, for every $b\in B$. Since the map $B \ni b\rightarrow \delta_2(b)\in B$ is a group homomorphism whose image has finite index in $B$, and thus is an ICC group, we can apply \cite[Lemma 7.1]{bv13} to find an element $k\in B$ with $\delta_2(b)=kbk^{-1}$, or $\delta_2(b)^{-1}=kb^{-1}k^{-1}$, for every $b\in B$. 

Hence, after replacing $\Delta$ with $\text{Ad}(1\otimes u_{k}\otimes 1\otimes 1\otimes 1)\circ\Delta$ and $\eta_b$ with $\text{Ad}(1\otimes u_{k}\otimes 1\otimes 1\otimes 1)(\eta_b)d(kb^{-1}k^{-1},k)d(k^{-1},kb^{-1})\in\cL(A^{(I)})\:\overline{\otimes}\:\cL(A^{(I)})^{\text{op}}\:\overline{\otimes}\:\cL(A^{(I)})\:\overline{\otimes}\:\mathbb{M}_n(\C)\:\overline{\otimes}\:\mathbb{M}_n(\C)^{\text{op}}$, we may assume $\delta_2=\text{Id}_B$, and 

\begin{equation*}
    \Delta(u_{\hat{b}})=\eta_b(u_{\hat{b}}\otimes \overline u_{\widehat{b^{-1}}}\otimes u_{\hat{b}}\otimes 1\otimes 1),\quad\text{for all }b\in B.
\end{equation*}\end{subproof}

\noindent To finish the proof, let $g\in G$. Let $b=\varepsilon(g)\in B$ and $a=g\widehat{b}^{-1}\in A^{(I)}$. Then

\begin{align*}
    \Delta(u_g)&=d(a,\widehat{b})^{-1}\Delta(u_a)\Delta(u_{\widehat{b}})=d(a,\widehat{b})^{-1}\Delta(u_a)\eta_b(u_{\widehat{b}}\otimes \overline u_{\widehat{b^{-1}}}\otimes u_{\widehat{b}}\otimes 1\otimes 1)\\
    &=d(a,\widehat{b})^{-1}d(a^{-1},g)^{-2}d(g^{-1},gd(b^{-1},b)^{-1}g^{-1}a)^{-1}\Delta(u_a)\eta_b(u_{a^{-1}}\otimes \overline u_{gd(b^{-1},b)^{-1}g^{-1}a}\otimes u_{a^{-1}}\otimes 1\otimes 1)\\
    &\quad\quad\quad(u_g\otimes \overline u_{g^{-1}}\otimes u_g\otimes 1\otimes 1).
\end{align*}

\noindent Thus, if we denote 

\begin{equation*}
    r_g=d(a,\widehat{b})^{-1}\Delta(u_a)\eta_bd(a^{-1},g)^2d(g^{-1},gd(b^{-1},b)^{-1}g^{-1}a)^{-1}(u_{a^{-1}}\otimes \overline u_{g d(b^{-1},b)^{-1}g^{-1}a}\otimes u_{a^{-1}}\otimes 1\otimes 1),
\end{equation*}

\noindent then $r_g\in \mathscr U(\cL(A^{(I)})\:\overline{\otimes}\:\cL(A^{(I)})^{\text{op}}\:\overline{\otimes}\:\cL(A^{(I)})\:\overline{\otimes}\:\mathbb{M}_n(\C)\:\overline{\otimes}\:\mathbb{M}_n(\C)^{\text{op}})$ and 

\begin{equation*}\label{w_g'}
    \Delta(u_g)=r_g(u_g\otimes \overline u_{g^{-1}}\otimes u_g\otimes 1\otimes 1), \quad \text{for every } g\in G.
\end{equation*}

\noindent Observe that for $g,h\in G$

\begin{equation*}
    \Delta(u_g)\Delta(u_h)=d(g,h)\Delta(u_{gh})=d(g,h)r_{gh}(u_{gh}\otimes \overline{u}_{(gh)^{-1}}\otimes u_{gh}\otimes 1\otimes 1),
\end{equation*}

\noindent while, on the other hand,

\begin{equation*}
    \Delta(u_g)\Delta(u_h)=r_g(\gamma_g\otimes\text{Id}_n\otimes\text{Id}_n^{\text{op}})(r_h)d(g,h)^2d(h^{-1},g^{-1})(u_{gh}\otimes \overline{u}_{h^{-1}g^{-1}}\otimes u_{gh}\otimes 1\otimes 1),
\end{equation*}

\noindent where $\gamma:G\curvearrowright \cL(A^{(I)})\:\overline{\otimes}\:\cL(A^{(I)})^{\text{op}}\:\overline{\otimes}\:\cL(A^{(I)})$ is the action given by $\gamma_g=\text{Ad}(u_g\otimes\overline{u}_{g^{-1}}\otimes u_g)$ built over the 2-cocycle $d\colon G\times G\to \mathbb{T}$ as in \eqref{actionwithcocycle}. In this case, and since $d(g,h)=\overline{d(h^{-1},g^{-1})}$, we obtain

\begin{equation*}
    r_{gh}=r_g(\gamma_g\otimes\text{Id}_n\otimes\text{Id}_n^{\text{op}})(r_h).
\end{equation*}

\noindent In other words, $r$ is a 1-cocycle for the action $\gamma_g\otimes\text{Id}_n\otimes\text{Id}_n^{\text{op}}$. Since the action $\gamma$ is built over $G\curvearrowright I$, by \cite[Theorem 3.6(a)]{cios22} there exists a unitary $s\in \cL(A^{(I)})\:\overline{\otimes}\:\cL(A^{(I)})^{\text{op}}\:\overline{\otimes}\:\cL(A^{(I)})\:\overline{\otimes}\:\mathbb{M}_n(\C)\:\overline{\otimes}\:\mathbb{M}_n(\C)^{\text{op}}$, a homomorphism $\xi\colon G\to\mathscr{U}_n(\C)\otimes\mathscr{U}_n(\C)^{\text{op}}$ and $\lambda_g\in\mathbb{T}$ such that $r_g=\lambda_gs^*(1\otimes 1\otimes 1\otimes\xi(g))(\gamma_g\otimes\text{Id}_n\otimes\text{Id}_n^{\text{op}})(s)$, for every $g\in G$.
Thus, after replacing $\Delta$ by $\text{Ad}(s)\circ\Delta$, we obtain  

\begin{equation}\label{xi_g}
    \Delta(u_g)=\lambda_g(u_g\otimes \overline u_{g^{-1}}\otimes u_g\otimes \xi(g)),\text{ for all }g\in G.
\end{equation}

\noindent Let $\mathscr{N}$ be the von Neumann algebra generated by $\{u_g\otimes \overline{u}_{g^{-1}}\otimes u_g\otimes x\otimes y:g\in G, x\in\mathbb{M}_n(\C), y\in\mathbb{M}_n(\C)^{\text{op}}\}$. Then, $\Delta(\cM)\subset\mathscr{N}\subset\Delta(\cM)(\mathbb{M}_n(\C)\:\overline{\otimes}\: \mathbb{M}_n(\C)^{\text{op}})$. By Lemma \ref{intertwiningpropofcom}(c), $\mathscr{N}=\Delta(\cM)$, and so $1\:\overline{\otimes}\:1\:\overline{\otimes}\:1\:\overline{\otimes}\:\mathbb{M}_n(\C)\:\overline{\otimes}\: \mathbb{M}_n(\C)^{\text{op}}\subset \Delta(\cM)$. In combination with \eqref{xi_g}, this implies $n=1$ and hence $t=1$. Also, $\xi(g)\in\mathbb{T}$ and 

\begin{equation*}
    \Delta(u_g)=\tilde{\lambda}_g(u_g\otimes \overline u_{g^{-1}}\otimes u_g),\quad\text{ for every } g\in G.
\end{equation*}

\noindent By Lemma \ref{positiveheight}, the height $h_H(G)>0$. Since $G$ is ICC, the unitary representation $(\text{Ad}(u_g))_{g\in G}$ of $G$ on $L^2(\mathcal M)\ominus\mathbb C1$ is weakly mixing, and since $H$ is ICC, $\cL_c(H)\not\prec\cL_c(C_H(k))$ for $k\neq e$. By applying \cite[Theorem 4.1]{DV24} we conclude that there exists a unitary $w\in\mathcal M$ and a group isomorphism $\rho\colon G\rightarrow H$ such that $u_g=\tilde{\lambda}_gwv_{\rho(g)}w^*$, for every $g\in G$.
\end{proof}

In light of our previous result, we conclude this section by proposing the following broader conjecture for study, which extends Popa's strengthening of Connes' Rigidity Conjecture \cite{popa07} to twisted II$_1$ factors associated with ICC property (T) groups.


\begin{conj}\label{twistedCRC} Let $G$ be an ICC property (T) group. Let $H$ be an arbitrary countable ICC group and let $d:G\times G\to\mathbb{T}$ and $c:H\times H\to\mathbb{T}$ be any 2-cocycles. Let $t>0$ and assume that $\Psi:\cL_d(G)^t\to\cL_c(H)$ is a $*$-isomorphism. 

Then $t=1$ and there exists a group isomorphism $\delta:G\to H$ for which $d$ is cohomologous to $c\circ \delta$. Moreover, there exists a map $\xi:G\to\mathbb{T}$ with $d(g,h)=\xi_g\xi_hc(\delta(g),\delta(h))\overline{\xi_{gh}}$, a multiplicative character $\eta:G\to\mathbb{T}$ and a unitary $w\in\cL_d(H)$ such that

\begin{equation*}
    \Psi(u_g)=\eta_g\xi_gwv_{\delta(g)}w^*,\text{ for all }g\in G.
\end{equation*}
\end{conj}

\section{Rigidity for property (T) central extensions}\label{section5}

\noindent This section is dedicated to proving the main results of the paper. Since we heavily rely on the results from \cite{cfqt23} regarding integral decomposition and the reconstruction of the center, we will use the same group notations from this paper. In particular, as in \cite[Definition 2.1]{cfqt23} and the remarks proceeding them, group extensions $1\ra Z(G)\ra G \overset{\e}{\ra} W\ra 1$ will be denoted by $G=Z(G)\rtimes_c W$ where $c\colon  W\times W\ra Z(G)$ is the $2$-cocycle satisfying $\sigma(g)\sigma(h)=c(g,h)\sigma(gh)$ for all $g,h\in W$ for $\sigma\colon  W\ra G$ a section with $\e\circ\sigma=id_W$. We encourage the interested reader to consult these notations and the results from \cite{cfqt23} beforehand.

We start by demonstrating  the following more precise statement from which Theorems A and B and \ref{TheoremC} will be derived.

\begin{thm}\label{mainresult} Let $C$ be a nontrivial free abelian group and let $D$ be a nontrivial ICC subgroup of a hyperbolic group with an action $D\curvearrowright I$ with amenable stabilizers. Let $W\in\mathcal{WR}(C,D\curvearrowright I)$ be any property (T) group. Assume that $G= A\rtimes_c W$ is any property (T) central extension with infinite center $A=Z(G)$ that splits over $C^{(I)}\leqslant W$. Let $H$ be any group and let $\Theta\colon  \cL(G)\ra \cL(H)$ be any $*$-isomorphism.

Then, $H\cong_{v}G$. More precisely, we can find a countable family of projections $\mathcal{P}\subset \cL(Z(H))$ with $\sum_{p\in\mathcal{P}}p=1$, finite subgroups $A_p<A$, $B_p<Z(H)$, a finite index subgroup $G_0\leqslant G$ containing $W$, $[G,G]$ and $A_p$ for all $p$, maps $\delta_p\colon G_0\to H$, for each $p\in\mathcal{P}$, a finite abelian group $\mathscr{V}$, a group homomorphism $r\colon W/[W,W]\to\mathscr{V}$, a character $\eta\colon W\to\mathbb{T}$, and a unitary $w\in\mathscr{U}(\cL(H))$ satisfying the following relations:

\begin{enumerate}
    \item The map $\Theta$ is given by 
    
    \begin{equation*}
        \Theta(u_g )=\eta_{\pi_1(g)}r_{\pi_2(g)}w\left(\sum_{p\in\mathcal{P}}    v_{\delta_p(g)}p\right)w^*, \quad \text{for all }  g\in G_0.
    \end{equation*}

    \noindent Here, $\pi_1\colon G_0\to W$ and $\pi_2\colon G_0\to W/[W,W]$ are the canonical quotient maps.

    \item For all $p \in \mathcal P$, $\delta_p (A_p)\leqslant B_p $ and $\text{Im}(\delta_p)$ is a finite index subgroup of $H$.
    \item For every $p$ there exists a subgroup  $[G,G] A_p \leqslant G_1\leqslant G_0$  such that the canonical map $\widehat\delta_p\colon  G_1/A_p \ra \text{Im}(\delta_p)B_p/B_p$ given by $\widehat \delta_p(g A_p)= \delta_p(g)B_p$ is a group isomorphism.
\end{enumerate}
\end{thm}

Before presenting the proof, we introduce some notation and preliminary results. Let $H$ be a countable discrete group. The \emph{FC-center} of $H$, denoted by $H^{fc}$, consists of all $g\in H$ whose conjugacy class $\mathscr O(g)=\{ hgh^{-1}\,:\, h\in H\}$ is finite. The upper FC-series of $H$ is the series $\{1\}=H^h_0\leqslant H^h_1\leqslant \cdots \leqslant H^h_n\leqslant\cdots$, where $H^h_{n+1}/H^h_n$ is the set of all FC-elements of $H/H^h_n$. Note that $H_1^h=H^{fc}$. This series stabilizes at some ordinal called the \textit{hyper-FC center} of $H$ which is denoted by $H^{hfc}$. By \cite[Proposition 2.2]{hfc}, $H/H^{hfc}$ is ICC.

To prove the previous theorem, we need an analogue of \cite[Theorem B]{cfqt23} for non-split central extensions. Since the proof is the same as the original, we present the statement here and recommend that interested readers consult the original result beforehand. 

\begin{thm}\label{almostcenter} Let $W\in \mathcal{WR}(C,D\curvearrowright I)$ be a property (T) group where $C$ is a non-trivial free abelian, $D$ is a non-trivial ICC subgroup of a hyperbolic group, and $D\curvearrowright I$ has amenable stabilizers. Let $G=A\rtimes_c W$ be any property (T) central extension with infinite center $A=Z(G)$ that splits over $C^{(I)}\leqslant W$. Assume that $H$ is an arbitrary group such that $\cL(G)\cong \cL(H)$.

Then $H^{fc}\leqslant H^{hfc}$ is finite index and the commutator $F:=[H^{fc},H^{fc}]$ is finite. Moreover, if we consider the central projection $z= |F|^{-1}\sum_{g\in F}  u_g \in \mathscr Z (\cL(H))$ then one can find a countable family of orthogonal projections $r_n \in \cL(H^{fc})z$ such that $\cL(H^{fc})z = \oplus_n \mathscr Z (\cL(H))r_n$.
\end{thm}

We now proceed to prove the main result of this paper, closely following the approach taken in \cite[Theorem A]{cfqt23}.

\begin{proof}[Proof of Theorem \ref{mainresult}] Let $H$ be a countable group and assume $\Theta(\cL(A\rtimes_cW))=\cL(H)=:\mathcal M$, where $A$ is infinite abelian and $W\in\mathcal{WR}(C,D\curvearrowright I)$ is a wreath-like product group. Notice that $H$ can be written as an extension of $H^{hfc}$ by $H/H^{hfc}$. To prove our result, we must first establish that $H/H^{hfc}$ is isomorphic to $W$, $Z(H)=H^{hfc}=H^{fc}$ and $Z(H)$ is virtually isomorphic to $A$. Denote by $F=[H^{fc},H^{fc}]$, and by $(X,\mu)$ (resp. $(Y,\nu)$) the probability spaces with $\mathscr{Z}(\cL(H))=L^{\infty}(X,\mu)$ (resp. $\cL(A)=L^{\infty}(Y,\nu)$). 

\vspace{2mm}

\begin{claim}\label{quotientiswreathlikeprod} $H/H^{hfc}\cong W$ and $H^{hfc}/F=Z(H/F)$.
\end{claim}

\begin{subproof}[Proof of Claim \ref{quotientiswreathlikeprod}] From Theorem \ref{almostcenter}, the commutator $F=[H^{fc},H^{fc}]$ is finite. Furthermore, observe that $H/F=H^{hfc}/F\rtimes_{\alpha,d}H/H^{hfc}$ is a group extension. Letting $z=|F|^{-1}\sum_{f\in F}u_f\in\mathscr{P}(\mathscr{Z}(\cL(H)))$ then \cite[Proposition 6.15]{cfqt23} yields the following cocycle cross-product decomposition $\cL(H)z\cong\cL(H^{hfc})z\rtimes_{\beta,\tilde{d}}H/H^{hfc}$. Here, $(\beta,\tilde d)$ is the von Neumann algebra cocycle action induced from the group cocycle action $(\alpha,d):H/H^{hfc}\curvearrowright H^{hfc}/F$, where $\tilde d\colon H/H^{hfc}\times H/H^{hfc}\to \mathscr U (\cL(H^{hfc}/F))$ is the $2$-cocycle induced by the group cocycle $d$, i.e.\ $\tilde d(g,h)=v_{d(g,h)}$ for all $g,h \in H/H^{hfc}$. In addition, we have $\cL(H^{hfc})z\cong\cL(H^{hfc}/F)$.

From Theorem \ref{almostcenter} one can find orthogonal projections $s_i\in \cL(H^{fc})z$ with $\sum_i s_i=z$ such that $\cL(H^{fc}/F)=\oplus_{i=1}^{\infty}\mathscr{Z}(\cL(H/F))s_i$. By \cite[Lemma 6.11]{cfqt23} the inclusion of abelian von  Neumann algebras $\mathscr Z(\cL(H))z \subseteq \cL(H^{hfc})z$ yields an integral decomposition $\cL(H^{hfc})z= \int^\oplus_{X_0} \mathcal D_x d\mu_0(x)$ where $\mathcal D_x$ are completely atomic for $\mu_0$-a.e.\ $x\in X_0$. Here, $X_0$ is a measurable space with $z=\mathbbm{1}_{X_0}$. Using \cite[Propositions 6.6 and 6.8]{cfqt23}, we further get

\begin{equation*}
    \mathcal M z=\cL(H)z\cong\int_{X_0}^{\oplus}\mathcal{D}_x\rtimes_{\beta_x,\tilde{d}_x}H/H^{hfc}\:d\mu_0(x),
\end{equation*}

\noindent Let $z= \Theta(p)\in \mathscr Z(\mathcal M)$ and note that $\cL(H)z= \Theta(\cL(A)p\rtimes_{\tilde{c} p}W)$, where $\tilde{c}p$ maps into the corner $\cL(A)p$; that is, $\tilde{c}p(g,h)=\tilde{c}(g,h)p$ for every $g,h\in W$. Then,

\begin{equation}\label{twistedfibersisom}
    \int_{Y_0}^{\oplus}\cL_{\tilde{c}p_y}(W)\:d\nu_0(y)=\cL(A)p\rtimes_{\tilde{c}p} W\cong\int_{X_0}^{\oplus}\mathcal{D}_x\rtimes_{\beta_x,\tilde{d}_x}H/H^{hfc}\:d\mu_0(x).
\end{equation}

\noindent By \cite[Theorem 2.1.14]{spaas}, the fibers are isomorphic; that is, there exists a Borel isomorphism $f\colon Y_0\to X_0$ with $\cL_{\tilde{c}p_{f^{-1}(x)}}(W)\cong\mathcal{D}_x\rtimes_{\beta_x,\tilde{d}_x}H/H^{hfc}$ for $\mu_0$-a.e.\ $x\in X_0$. In particular, $\mathcal{D}_x\rtimes_{\beta_x,\tilde{d}_x}H/H^{hfc}$ is a II$_1$ factor, and hence, by \cite[Lemma 2.5]{cfqt23}, $\mathcal{D}_x$ is finite dimensional. From \cite[Proposition 2.6]{cfqt23}, we find a subgroup $H_x\leqslant H/H^{hfc}$ of index $n_x\in\mathbb{N}$, $l_x\in\mathbb{N}$ and a 2-cocycle $\eta_x\colon H_x\times H_x\to\mathbb{T}$ satisfying $\mathcal{D}_x=\mathbb{M}_{l_x}(\mathbb C)\:\otimes\:\mathbb{D}_{n_x}$ and $\mathcal{D}_x\rtimes_{\beta,\tilde{d}_x}H/H^{hfc}\cong\cL_{\eta_x}(H_x)^{n_xl_x}$. 

By \eqref{twistedfibersisom}, $\cL_{\tilde{c}p_{f^{-1}(x)}}(W)\cong\cL_{\eta_x}(H_x)^{n_xl_x}$ and \ref{strongsuperrig} implies $n_xl_x=1$ and $W\cong H_x=H/H^{hfc}$. Moreover, since $n_x,l_x=1$, $\mathcal{D}_x=\C$ and $\cL(H^{hfc}/F)=\mathscr{Z}(\cL(H/F))$, we have $H^{hfc}/F=Z(H/F)$.
\end{subproof}

\begin{claim}\label{finiteconjiscenter} $H^{hfc}=Z(H)$.
\end{claim}

\begin{subproof}[Proof of Claim \ref{finiteconjiscenter}] Using \cite[Theorem 2.8]{cfqt23} we can decompose 

\begin{equation}\label{decompLH}
    \cL(H)\cong\bigoplus_{i=1}^k\mathbb{M}_{n_i}(\C)\otimes(\mathbb{D}_{m_i}\rtimes_{\alpha_i,s_i}((H^{hfc}/F)\rtimes_r H/H^{hfc})).
\end{equation}

\noindent In addition, we write $\mathbb{D}_{m_i}=\bigoplus_{t=1}^{t_i}\mathbb{D}_{s_i}^t$ for which there exist finite index subgroups $A_t^i$ of $H^{hfc}/F$ acting transitively on $\mathbb{D}_{s_i}^t$. Moreover, we obtain unitaries $x_{t,a}^i\in\mathbb{D}_{s_i}^t$ and a finite index subgroup $A_0\leqslant H^{hfc}/F$ such that

\begin{equation*}
    \mathscr{Z}(\cL(H))\cong\left\{\left(\sum_ix_a^i\right)a:a\in A_0\right\}\subseteq \mathscr{Z}(\cL(F\rtimes_{\alpha,w}H^{hfc}/F))\cong \bigoplus_i\bigoplus_j\cL_{\eta_t^i}(A_t^i).
\end{equation*}

\noindent Denote by $\mathcal{A}=\cL(F\rtimes_{\alpha,w}H^{hfc}/F)$. Taking the integral decomposition of $\cL(H)$ over its center and by \cite[Proposition 6.8]{cfqt23}, we obtain

\begin{equation*}
    \cL(H)\cong\int^\oplus_X\mathcal{A}_x\rtimes_{\beta_x,\gamma_x}H/H^{hfc}\:d\mu(x).
\end{equation*}

\noindent Here, the von Neumann algebras $(\mathcal{A}_x)_{x\in X}$ are the decomposition of $\mathcal{A}$ over $\mathscr{Z}(\cL(H))\subseteq \mathcal{A}$. On the other hand, since $\cL(H)\cong \cL(A\rtimes_cW)$, taking the decomposition over the center $\cL(A)\cong\mathscr{Z}(\cL(H))$, one can find a measure preserving map $f\colon Y\to X$ satisfying

\begin{equation}\label{isomfibers}
    \cL_{\tilde{c}_{f^{-1}(x)}}(W)\cong\mathcal{A}_x\rtimes_{\beta_x,\gamma_x}H/H^{hfc},\text{ for $\mu$-a.e. }x\in X. 
\end{equation}

\noindent Since this decomposition is over the center, the von Neumann algebras $\mathcal{A}_x\rtimes_{\beta_x,\gamma_x}H/H^{hfc}$ are factors for $\mu$-a.e.\  $x\in X$. Notice the inclusion $\mathscr{Z}(\cL(H))\subseteq \cL(F\rtimes_{\alpha,w}H^{hfc}/F)$ is of finite index. Thus, using \cite[Theorem 6.12]{cfqt23}, $\mathcal{A}_x$ are finite dimensional algebras for $\mu$-a.e.\ $x\in X$. By \cite[Proposition 2.6]{cfqt23}, there exists $l_x,n_x\in\N$ such that $\mathcal{A}_x=\mathbb{M}_{l_x}\:\otimes\:\mathbb{D}_{n_x}$, a subgroup $H_x\leqslant H/H^{hfc}$ of finite index $n_x$ and a 2-cocycle $\eta_x\colon H_x\times H_x\to\mathbb{T}$ satisfying 

\begin{equation}\label{cocycrosstwisted}
    \mathcal{A}_x\rtimes_{\beta_x,\gamma_x}H/H^{hfc}\cong \cL_{\eta_x}(H_x)^{n_xl_x}.
\end{equation}

\noindent Combining isomorphisms \eqref{isomfibers} and \eqref{cocycrosstwisted}, together with \ref{strongsuperrig}, we obtain that $n_x=l_x=1$. This shows, in particular, that $\mathcal{A}_x=\C$ for $\mu$-a.e. $x\in X$ and hence,

\begin{equation*}
    \cL(H^{hfc})\cong\cL(F\rtimes_{\alpha,t}H^{hfc}/F)\cong\int_X\mathcal{A}_xd\mu(x)\cong\mathscr{Z}(\cL(H)).
\end{equation*}

\noindent Therefore, $H^{hfc}=Z(H)$ and $F=\{1\}$.\end{subproof}

Altogether, Claims \ref{quotientiswreathlikeprod} and \ref{finiteconjiscenter} yield $H\cong Z(H)\rtimes_d H/Z(H)\cong Z(H)\rtimes_d W$. To recover the center, we examine the isomorphism $\Theta\colon \cL(A)\rtimes_{\tilde{c}}W\to \cL(Z(H))\rtimes_{\tilde{d}}H/Z(H)$ in its integral decomposition form

\begin{equation*}
    \int_Y\cL_{\tilde{c}_y}(W)d\nu(y)\cong\int_X\cL_{\tilde{d}_x}(H/Z(H))d\mu(x).
\end{equation*}

\noindent By \cite[Theorem 2.1.14]{spaas} there exists full measure sets $X'\subset X$ and $Y'\subset Y$, a Borel isomorphism $f\colon X'\to Y'$ with $f(\mu)$ equivalent to $\nu$, and a measurable field $y\mapsto \Theta_y$ of tracial isomorphisms $\Theta_y
\colon \cL_{\tilde{c}_y}(W) \to \cL_{\tilde{d}_{f^{-1}(y)}}(H/Z(H))$ such that 

\begin{equation*}
    \Theta=\int_Y^{\oplus}\Theta_yd\nu(y).
\end{equation*}

\noindent Using \ref{strongsuperrig} one can find a group isomorphism $\rho_y\colon W\to H/Z(H)$, a map $\xi^y\colon W\to \mathbb{T}$ satisfying 

\begin{equation*}\label{fiberscohomologous1}
    \tilde{d}_{f^{-1}(y)}(\rho_y(g),\rho_y(h))=\overline{\xi^y_g}\overline{\xi^y_h} \xi^y_{gh} \tilde{c}_y(g,h)\text {, for all }g,h \in W,
\end{equation*} 

\noindent a unitary $w_y\in \cL_{\tilde{d}_{f^{-1}(y)}}(H/Z(H))$ and a character $\eta^y\colon W\to\mathbb{T}$ such that 

\begin{equation*}        
    \Theta_y(u_g) = \eta^y_g\xi^y_{g} w_y v_{\rho_y(g)}w_y^*\text{, for all }g\in W.
\end{equation*} 

\noindent Next we will combine the arguments from \cite[Proposition 6.9 and Theorem A]{cfqt23} to describe the map $\Theta$. Since $W$ has property $(T)$, there are at most countably many group isomorphisms, denoted by $(\rho_n)_{n\in\N}$. Consider the set 

\begin{align*}
    Y_n:=\{(y,&(\xi_g)_{g\in W},(\eta_g)_{g\in W},w)\in Y\times\mathscr{B}(\ell^2(W,H/Z(H)))\times\mathscr{B}(\ell^2(W,H/Z(H)))\times\mathscr{B}(\ell^2(H/Z(H))):\\
    &\xi_g,\eta_g\in\mathbb{T}_{f^{-1}(y)}, \text{ for all }g\in W,\:\eta\text{ is a group homomorphism,}\\
    &w\in\mathscr{U}(\cL_{\tilde{d}_{f^{-1}(y)}}(H/Z(H)))\text{ and }\Theta_y(u_g)w=\eta_g\xi_gwv_{\rho_n(g)}\text{ for all }g\in W\}.
\end{align*}

\noindent Since $y\mapsto\Theta_y(u_g)$ and $y\mapsto v_{\rho_n(g)}$ are measurable fields of operators, $Y_n$ is a Borel set. Observe that $\pi_1(Y_n)\subseteq Y$, where $\pi_1$ is the projection onto the first coordinate. By \cite[Theorem A.16]{takesaki} there exist measurable crossed sections $s_n^g,r_n^g\colon \pi_1(Y_n)\to \mathscr{B}(\ell^2(W,H/Z(H)))$, for each $g\in W$, and $w_n\colon \pi_1(Y_n)\to\mathscr{B}(\ell^2(H/Z(H)))$, such that $(y,s_n^g(y),r_n^g(y),w_n(y))\in Y_n$ for all $y\in\pi_1(Y_n)$. 

By the prior paragraph, we see that $\{\pi_1(Y_n)\}_n$ forms a measurable partition of $Y$. Hence, letting $e_n=\chi_{\pi_1(Y_n)}\in\mathscr{P}(\cL(Z(H)))$, $w=\int_Y^{\oplus}w_yd\nu(y)$ where $w_y=w_n(y)$ for $y\in\pi_1(Y_n)$, $\xi_g=\int_Y^{\oplus}s_y^gd\nu(y)$ where $s_y^g=s_n^g(y)$ for $y\in \pi_1(Y_n)$, and $\eta_g=\int_Y^{\oplus}r_y^gd\nu(y)$ where $r_y^g=r_n^g(y)$ for $y\in \pi_1(Y_n)$, we obtain

\begin{equation}\label{isomformula}
    \Theta(u_g)=\int_Yr_y^gs_y^gw_yv_{\rho_y(g)}w_y^*d\nu(y)=\eta_g\xi_gw\left(\sum_ne_nv_{\rho_n(g)}\right)w^*,\text{ for all }g\in W.
\end{equation}

\noindent In particular, we have  

\begin{equation}\label{cornercohomologous}
    \Theta(\tilde{c}(g,h))=\xi_g\xi_h\xi_{gh}^*\left( \sum_n\tilde{d}(\rho_n(g),\rho_n(h))e_n\right ),\text{ for all }g,h\in W.
\end{equation}

\noindent From \eqref{cornercohomologous} the $\ast$-isomorphism $\Theta\colon \cL(A)\rtimes_{\tilde{c}}W\to \cL(H)$ satisfies

\begin{equation}\label{coh2cocycl}
    \Theta(u_{c(g,h)})=\xi_g\xi_h\xi_{gh}^*\left(\sum_n v_{d(\rho_n(g),\rho_n(h))}e_n \right),\quad\text{for all }g,h\in W.
\end{equation}

\noindent Let $\mathscr{U}$ denote the group of all unitaries in $\cL(Z(H))$. Since $Z(H)$ is abelian then so is $\mathscr U$ and therefore $\mathscr A:=\Theta(A)(\sum_n Z(H)e_n)\leqslant \mathscr{U}$ is a normal subgroup. Denote by $\widehat{\mathscr{U}}:=\mathscr{U}/{\mathscr A}$. From relation \eqref{coh2cocycl}, $\xi_{gh} \Theta(u_{c(g,h)})=\xi_g\xi_h\left(\sum_n v_{d(\rho_n(g),\rho_n(h))}e_n \right)$ for all $g,h\in W$, and viewing this equation in $\widehat{\mathscr{U}}$, we have

\begin{equation*}
    \widehat{\xi_{gh}}=\widehat{\xi_g}\widehat{\xi_h},\quad\text{for all }g,h\in W.
\end{equation*}

\noindent Hence, the map $\widehat{\xi}\colon W\to \widehat{\mathscr{U}}$ given by $g\mapsto \widehat{\xi_g}$ is a group homomorphism. Since $\widehat{\mathscr{U}}$ is abelian, $\widehat{\xi_g}=1$ for all $g\in[W,W]$. This means the map $\widehat{\xi}$ factors through $[W,W]$, say $r\colon W/[W,W]\to\widehat{\mathscr{U}}$. Since $W$ has property (T), the image of $r$ is a finite subgroup of $\widehat{\mathscr{U}}$. Let $\pi\colon W\to W/[W,W]$ be the canonical quotient map. Since $\xi_g\in r_{\pi(g)}\Theta(A)(\sum_nZ(H)e_n)$, one can find maps $k\colon W\to A$ and $l^n\colon W\to Z(H)$, for $n\in \mathbb N$ satisfying

\begin{equation}\label{formulaxi}
    \xi_g=r_{\pi(g)}\Theta(u_{{k_g}^{-1}})\left(\sum_nv_{l^n_g}e_n\right ) \text{ for all }g\in W.
\end{equation}

\noindent Substituting these formulae in \eqref{coh2cocycl} gives

\begin{equation}\label{ThetaonLA}
    \Theta(u_{k_gk_hk_{gh}^{-1}c(g,h)})=\sum_n v_{l^n_gl^n_h(l^n_{gh})^{-1}d(\rho_n(g),\rho_n(h))}e_n,\quad\text{for all }g,h\in W.
\end{equation}

\noindent Moreover, \eqref{formulaxi} and \eqref{isomformula} imply that 

\begin{equation}\label{formulahom2}
    \Theta (u_{k_g g} )=\eta_gr_{\pi(g)}w  \left (\sum_n v_{l^n_g \rho_n(g)} e_n\right) w^*, \text{ for all }g \in W.
\end{equation}

\noindent Next we observe that the group $2$-cocycles $t(g,h):=k_gk_hk_{gh}^{-1}c(g,h) \in A$ and  $s^n(\rho_n(g),\rho_n(h)):=l^n_gl^n_h(l^n_{gh})^{-1}d(\rho_n(g),\rho_n(h)))\in Z(H)$ are cohomologous to $c $ and $d$, respectively. Let $\tilde{g}=k_gg$ for $g\in W$ and note that $W\cong \tilde{W}$ and $G=Z\rtimes_t\tilde{W}$. Let $A_0$ be the subgroup of $A$ generated by the image of $t$. The formulae \eqref{ThetaonLA} and \eqref{formulahom2} yield maps $\delta_n\colon A_0\rtimes_{t}\tilde{W}\to Z(H)\rtimes_{s^n}H/Z(H)=H$ with 

\begin{equation*}
    \Theta(u_g) = \eta_{\pi_1(g)}r_{\pi_2(g)}w \left(\sum_nv_{\delta_n(g)}e_n\right)w^*,
\end{equation*}

\noindent for all $g\in A_0\rtimes_t \tilde{W}$ and for all $n\in \mathbb N$. Here $\pi_1\colon A_0\rtimes_t\tilde{W}\to \tilde{W}$ is the canonical quotient map. From construction one can check that ${\rm Im}(\rho_n)\subset {\rm Im}(\delta_n)$, for all $n\in \mathbb N$. This shows part (1) of the theorem.

To argue for the other two parts, recall that for all $g\in [\tilde{W},\tilde{W}]$ we have  $r_g,\eta_g=1$. Let $A'\leqslant A_0$ be the group generated by the image of the cocycle $t$ on $[\tilde{W},\tilde{W}]$ and consider $G_1:=A'\rtimes_t[\tilde{W},\tilde{W}]$. Note that $G_1$ is finite index in $G$ by Lemma \ref{finiteindex[W,W]}. If we denote by $(f_n)_n \subset \cL(A)$ the projections satisfying  $\Theta(f_n)=e_n$,  for all $n\in \mathbb  N$, then the maps $\delta_n\colon G_1\to H$ satisfy $\Theta(u_gf_n)=wv_{\delta_n(g)}e_nw^*$ for all $g\in G_1$ and for all $n\in\N$. Summing these relations we get

\begin{equation}\label{isomon[W,W]}
    \Theta (u_g)= w \left ( \sum_n v_{\delta_n(g)} e_n \right ) w^*\text{ for all }g\in G_1.
\end{equation}

\noindent Consider the subgroups $A_n=\{g\in A':u_gf_n=f_n\}$, $B_n=\{h\in Z(H):v_he_n=e_n\}$ and $H_n=\{h\in H:\text{there exists }g\in G_1 \text{ with }\Theta(u_gf_n)=wv_he_nw^*\}$. Note that $A_n$ and $B_n$ are finite. Moreover, one can check that $\delta_n(A_n)\leqslant B_n$ and also $H_n=\text{Im}(\delta_n)B_n$ for all $n\in \mathbb N$. The last assertion follows because for $h\in H_n$, $v_{\delta_n(g)}e_n=v_he_n$, or $h^{-1}\delta_n(g)\in B_n$. Additionally, the group $H_n$ satisfies $\cL(H_n)e_n=\Theta(\cL(G_1))e_n$, which implies that $H_n$ is a finite index subgroup of $H$. 

Finally, the induced map $\widehat{\delta_n}\colon G_1/A_n\to \text{Im}(\delta_n)B_n/B_n$ given by $\widehat{\delta_n}(gA_n)=\delta_n(g)B_n$ is an isomorphism. This combined with the previous relations conclude parts (2) and (3) of the theorem. \end{proof}

In the remaining part, we will demonstrate \ref{TheoremA}, \ref{theoremB} and \ref{TheoremC}, assuming the same setup and using the same notations and relations from the proof of Theorem \ref{mainresult}. 

\begin{proof}[\textbf{\textbf{Proof of \ref{TheoremA}}}] The conclusion of Theorem \ref{mainresult}, implies that $\widehat{\delta_n}\colon G_1/A_n\to H_n/B_n$ is a group isomorphism; hence $G\cong_v H$, as desired. \end{proof}

\begin{proof}[\textbf{Proof of \ref{theoremB}}] First, since $G$ has trivial abelianization, so has $\tilde{W}$.  Therefore the map $\Theta$ is given by relation \eqref{isomon[W,W]} on $G_1$. Moreover, the trivial abelianization of $G$ implies $A'=A_0=A$, and so $G_1=G$. Second, since $G$ is torsion free, $A_n=\{1\}$ for all $n$ so $\widehat{\delta_n}\colon G\to H/B_n$ is an injective group homomorphism onto $H_n/B_n$. In particular, $H_n/B_n$ has trivial abelianization for all $n$. 

Fix $n\in\N$. Note that $[gB_n,hB_n]=[g,h]B_n$ for all $g,h\in\text{Im}(\delta_n)$. Thus, for $a\in \text{Im}(\delta_n)$, there exists $g_1,...,g_l\in[\text{Im}(\delta_n),\text{Im}(\delta_n)]$ with $aB_n=g_1g_2\cdots g_lB_n$. This means that $a\in[\text{Im}(\delta_n),\text{Im}(\delta_n)]B_n$ and hence, $H_n=[\text{Im}(\delta_n),\text{Im}(\delta_n)]B_n\subseteq [H,H]B_n$. Moreover, since $H/H_n$ is abelian, we obtain $H_n=[H,H]B_n$. Observe that for $n,m\in\N$, we have

\begin{equation*}
    G\cong[H,H]B_n/B_n\cong [H,H]/([H,H]\cap B_n)\geqslant ([H,H]\cap B_m)/(([H,H]\cap B_n)\cap B_m).
\end{equation*}

\noindent Since $G$ is torsion free, each $B_m$ is finite, and $n,m$ are arbitrary, we have $[H,H]\cap B_n=[H,H]\cap B_m$ for all $n,m\in\N$. Next, we show the collection of $\{B_n\}_n$ is finite. Otherwise, $\{[H,H]B_n\}_n$ is an infinite collection of finite index subgroups of $H$ each containing $[H,H]$. Note that $[H,H]\leqslant H$ is finite index, as $H$ has property (T). Since $H$ is finitely generated, it can only have finitely many subgroups of a given finite index, meaning that the collection $\{H_n\}_n$ cannot be infinite. Thus, we obtain there are only finitely many $B_n$'s and $H_n$'s.
\end{proof}

\begin{proof}[\textbf{Proof of \ref{TheoremC}}] As $G$ has trivial abelianization, we have $G= G_1$. Since $\text{Out}(G)=\{1\}$, relation \eqref{isomon[W,W]} implies $n=1$ and the existence of a unique $\delta_1\colon G \to H$ with

\begin{equation*}
    \Theta(u_g)=wv_{\delta_1(g)}w^*,\text{ for all }g\in G.
\end{equation*}

\noindent Moreover, $A_1,B_1=\{1\}$ and hence $\delta_1$ is a group isomorphism.
\end{proof}

\vspace{2mm}

\subsection{A broader open problem}\label{section5.1} Our results demonstrate that the groups constructed in \ref{TheoremD} and \ref{TheoremC} satisfy the following natural extension of Connes' rigidity conjecture:

\begin{conj} Let $G$ be any property (T) group whose central quotient $G/Z(G)$ is an ICC group. Then for any group $H$ satisfying $\cL(G)\cong\cL(H)$ we have $G\cong_vH$.
\end{conj}

Stating this conjecture for groups $G$ with infinite centers and ICC central quotient $G/Z(G)$ is a natural step to broaden this study, which in some sense is the closest possible extension of Connes' original conjecture. Indeed, as prior works \cite{cfqt23,DV24center} and the current paper emphasize, proving W$^*$-superrigidity results for these groups relies first on fairly general techniques for reconstructing their center through classifying subalgebras that satisfy weak compactness conditions. This, in turn, enables one to reduce our conjecture to establishing W$^*$-superrigidity results for the fibers, which, in this case, are the twisted group von Neumann algebras of the ICC property (T) central quotient $G/Z(G)$. This is nothing other than the twisted version of Connes' rigidity conjecture (see Conjecture \ref{twistedCRC}). This suggests that proving or disproving our conjecture is as difficult as proving or disproving Connes' original conjecture. 

However, nothing prevents one from trying to pursue our conjecture for the entire class of property (T) groups. In this generality, it would be more appropriate to formulate the following.  

\begin{ques}\label{crcgf} Let $G$ be any property (T) group. Is it true that whenever $H$ is an arbitrary group satisfying $\cL(G)\cong \cL(H)$ then we have $G\cong_v H$? \end{ques}

Since general property (T) groups $G$ can have a much more complicated structure \cite{Ershov} than the subclass of central extensions with ICC central quotient, it seems less likely that this has a positive answer in its full generality. Specifically, reducing this rigidity to the study of simpler structures, as in our cases, is rather challenging. While we suspect the answer for  Question \ref{crcgf} is negative, we currently do not have any concrete counterexamples.

Bekka described in \cite{bekka21} the fibers of the integral decomposition of $\mathcal{L}(G)$ as von Neumann algebras arising from induced representations over the FC-center $G^{fc}$; that is, 

\begin{equation*}
    \cL(G)=\int_X^{\oplus}\text{Ind}_{G^{fc}}^G\pi_x(G)''d\mu(x).
\end{equation*}

\noindent Thus, for property (T) groups $G$ whose FC-radical is highly non-abelian (e.g., subgroups of $\mathfrak{S}_\infty$), this would require a new structural study of algebras that have not been considered in this setting before. Moreover, even attempting to leverage the alternative description of the fibers as cocycle crossed product von Neumann algebras (see \cite[Theorem 6.8]{cfqt23}) seems technically quite difficult. Thus, we believe that even the mere problem of identifying examples of property (T) groups $G$ whose $G^{fc}$ is nonvirtually abelian and locally finite, for which Question \ref{crcgf} has a positive answer, would represent a significant advancement.

Manufacturing a counterexample to Question \ref{crcgf} is closely related to the problem of identifying flexible groups; these are groups 
$G$ for which there exist many non-isomorphic groups $G\ncong H$ such that $\mathcal{L}(G) \cong \mathcal{L}(H)$. At present, every known example of flexible groups is derived from constructions involving infinite abelian groups, ICC amenable groups, or any canonical combinations of these. Motivated by this observation, and as highlighted by Stefaan Vaes recently at the workshop BIRS 24w5174, we propose the following open problem:

\begin{prob} Identify new flexible groups beyond the class of infinite abelian groups, ICC amenable groups, and their combinations. 
\end{prob}

\printbibliography 

\end{document}